\documentclass[amsfonts,11pt]{amsart}
\usepackage{amsfonts, mathrsfs}
\usepackage{ifthen}
\usepackage{amsthm}
\usepackage{amsmath}
\usepackage{graphicx}
\usepackage{amscd,amssymb,amsthm}
\usepackage{graphicx}
\usepackage{epstopdf}
\usepackage{hyperref}
\usepackage{algorithm2e}
\usepackage{mathtools}

\newcounter{minutes}
\divide\time by 60
\newcounter{hours}
\multiply\time by 60 \addtocounter{minutes}{-\time}

\newcommand{\real}{\operatorname{Re}}

\newcommand{\D}{\mathbb D}

\newcommand{\Z}{\mathbb Z}

\newtheorem{lemma}{Lemma}

\newtheorem{remark}{Remark}

\newtheorem{theorem}{Theorem}

\usepackage{xcolor}
\begin{document}

\title[Tur\'an type inequalities for oscillatory special functions]{Tur\'an type inequalities for oscillatory special functions}

\author[I. Akta\c{s}]{Ibrahim Akta\c{s}}
\address{Department of Mathematics, Karamano\u{g}lu Mehmetbey University, 70100 Karaman, Turkey}
\email{ibrahimaktas@kmu.edu.tr}

\author[\'A. Baricz]{\'Arp\'ad Baricz$^{\bigstar}$}
\address{Department of Economics, Babe\c{s}-Bolyai University, 400591 Cluj-Napoca, Romania and Institute of Applied Mathematics, \'Obuda University, 1034 Budapest, Hungary}
\email{bariczocsi@yahoo.com}

\def\thefootnote{}
\footnotetext{ \texttt{File:~\jobname .tex,
         printed: \number\year-0\number\month-\number\day,
          \thehours.\ifnum\theminutes<10{0}\fi\theminutes}
} \makeatletter\def\thefootnote{\@arabic\c@footnote}\makeatother

\dedicatory{Dedicated to Ahsen, Bor\'oka, Hafsa and Kopp\'any}

\thanks{${}^{\bigstar}$Corresponding author}

\maketitle

\begin{abstract}
In this paper our aim is to prove a conjecture of \'A. Baricz on Bessel functions of the first kind, which improves the classical Tur\'an type inequality for Bessel functions of the first kind proved by O. Sz\'asz. The idea is to consider the normalized Tur\'an expression for Bessel functions of the first kind and to show that between two consecutive zeros of the Bessel functions of the first kind the branch of this normalized Tur\'an expression has a local minimum, and the minimum values form a strictly increasing convergent sequence, whose terms are strictly decreasing with respect to the order. Moreover, we extend this result to general Bessel functions and regular Coulomb wave functions, and we use a similar approach to show sharp Tur\'an type inequalities for modified Bessel functions of purely imaginary order and parabolic cylinder functions. In addition, we prove a complex analogue of the result on Bessel functions of the first kind about the strictly decreasing property of the successive minimum values: an inclusion property in the complex plane of the Tur\'an expression for Bessel functions of the first kind by using the subordinating factor sequence technique in the sense of H.S. Wilf. The techniques employed in the paper may be useful to treat similar problems where Tur\'anians or normalized Tur\'anians of other oscillatory special functions appear.
\end{abstract}

\section{\bf Introduction}

Tur\'an-type inequalities form a classical and still active part of the theory of special functions, orthogonal polynomials, and analytic inequalities. For a family of functions or polynomials $\{u_\lambda\}$ indexed by an integer or a real parameter, the basic object is the Tur\'anian
\[\Delta_\lambda(x)=u_\lambda^2(x)-u_{\lambda-1}(x)u_{\lambda+1}(x),\]
and the central question is to determine its sign, its sharp bounds, and the parameter and argument ranges on which these properties hold. When the members of the family are positive, such inequalities are closely related to log-concavity or log-convexity in the index, for oscillatory functions, however, the determinant formulation is more robust because the individual functions change sign. The subject originates with the Tur\'an's inequality for the Legendre polynomials. P. Tur\'an communicated his result on Legendre polynomials to G. Szeg\H{o} \cite{Szego1948}, who published four different proofs in 1948 and discussed extensions to other classical families of orthogonal polynomials, however, Tur\'an's own paper \cite{Turan1950}, motivated by the zeros of Legendre polynomials, appeared in 1950. This starting point already displays two features that have remained central up to now: the role of three-term recurrence relations and the relation between Tur\'an determinants and the zero geometry of the underlying solutions. G. Szeg\H{o}'s work was quickly followed by results for ultraspherical, Laguerre, and Hermite polynomials, and by systematic treatments based on recurrence formulas and entire-function methods, see for example \cite{Szasz1950,Szasz1951,ThiruvenkatacharNanjundiah1951,Skovgaard1954}. For Jacobi polynomials, G. Gasper obtained an important normalized Tur\'an inequality under natural restrictions on the parameters \cite{Gasper1971,Gasper1972}. Moreover, at a more structural level, S. Karlin and G. Szeg\H{o} studied higher-order determinants with orthogonal-polynomial entries, placing the second-order Tur\'anian in a wider theory of sign-regular determinants \cite{KarlinSzego1960}. These developments made clear that Tur\'an-type inequalities are not just isolated identities but manifestations of positivity, total positivity, and zero interlacing in families satisfying differential or recurrence equations.

The oscillatory case is particularly significant because positivity arguments based only on the values of the functions are generally unavailable across successive nodal intervals. Bessel functions provided one of the earliest and most influential examples beyond orthogonal polynomials: O. Sz\'asz \cite{Szasz1950} treated Bessel functions together with ultraspherical polynomials in 1950, and V.R. Thiruvenkatachar and T.S. Nanjundiah \cite{ThiruvenkatacharNanjundiah1951} developed a broad collection of related Bessel and orthogonal-polynomial inequalities in 1951. For Bessel functions of the first kind, Tur\'an determinants are closely connected with recurrence relations, derivative identities, Neumann-series representations, and the distribution of real zeros. A survey and substantial extension, including higher-order Tur\'an determinants and results for Bessel functions of the second kind, was given by \'A. Baricz and T.K. Pog\'any \cite{BariczPogany2014}. The passage from $J_\mu$ and $Y_\mu$ to a general cylinder function $C_\mu=(\cos\alpha)J_\mu-(\sin\alpha)Y_\mu$ shows especially clearly how the sign of the Tur\'anian interacts with the first positive zero and with the oscillatory region \cite{BariczPonnusamySingh2016}. In this setting, sharp inequalities are often local in the argument or depend explicitly on the zeros, in contrast with the simpler global positivity statements available for many non-oscillatory functions. The same philosophy has been successfully extended to other solutions of Bessel-type equations. For Lommel functions, \'A. Baricz and S. Koumandos \cite{BariczKoumandos2016} used canonical products, information on real zeros, and the Laguerre--P\'olya class to obtain Tur\'an type inequalities. For Struve functions, \'A. Baricz, S. Ponnusamy, and S. Singh \cite{BariczPonnusamySingh2017} combined zero information with Mittag--Leffler and infinite-product representations, again emphasizing the effectiveness of methods adapted to oscillatory entire functions. These works illustrate a general principle: when a special function has sufficiently controlled real zeros, the logarithmic derivative obtained from a product expansion can often be converted into a Tur\'an-type inequality. Conversely, a Tur\'an type inequality can yield monotonicity information for logarithmic derivatives and ratios of contiguous functions, thereby feeding back into the analysis of zeros and extrema.

Modified Bessel functions form a useful complementary, non-oscillatory model in which the connection with parameter log-concavity and log-convexity can be made globally.
The classical inequality for $I_\nu$ was already obtained by V.R. Thiruvenkatachar and T.S. Nanjundiah \cite{ThiruvenkatacharNanjundiah1951}, while sharp two-sided estimates for the Tur\'anians of $I_\nu$ and $K_\nu$ were developed much later \cite{Baricz2010}. Such estimates are frequently equivalent, through recurrence relations, to bounds for the logarithmic derivatives or for ratios of functions of consecutive orders. This equivalence is important in computation: J. Segura \cite{Segura2011} showed that sharp ratio bounds lead to effective estimates for associated Tur\'an type inequalities. Thus, Tur\'an type inequalities have a practical role in stable evaluation and approximation of special-function ratios, not merely a qualitative role in proving positivity. They also have applications outside classical analysis. For example, bounds equivalent to Tur\'an inequalities for modified Bessel functions have been used to prove log-concavity properties of the gamma-gamma distribution function \cite{BariczPonnusamyVuorinen2011}. Related monotonicity results for products of modified Bessel functions enter the stability analysis of radially symmetric solutions of a generalized FitzHugh--Nagumo equation \cite{BariczPonnusamy2013}. More generally, the same estimates occur in problems involving applications in applied mathematics, biology, chemistry, physics and
engineering sciences, and numerical evaluation whenever ratios or logarithmic derivatives of special functions control the quantity of interest, see \cite{Baricz2015} for more details. Related to the results on non-oscillatory special functions we also refer to the papers \cite{KarpSitnik2010} and \cite{KalmykovKarp2013} (and to the references therein) of D. Karp and his coauthors on hypergeometric functions.

Accordingly, the modern theory combines several complementary tools: three-term recurrences, differential equations and Riccati transformations, integral representations, Sturm comparison, zero interlacing, canonical products, and the Laguerre-P\'olya theory of real entire functions. For oscillatory special functions, methods that retain explicit information on zeros are often indispensable, because the Tur\'anian may have a qualitatively different behavior before and after the first few zeros. The purpose of the present work is situated in this line of research: to establish sharp Tur\'an type inequalities for special functions with oscillatory structure in order to improve and complement the known results in the literature for Bessel functions of the first kind, general cylinder functions, regular Coulomb wave functions, modified Bessel functions of the second kind of purely imaginary order and parabolic cylinder functions. The paper is organized as follows: in Section 2 our aim is to solve a conjecture on Bessel functions of the first kind proposed by \'A. Baricz in \cite{Baricz2010} and to complete the result on the successive minima of the normalized Tur\'anian for Bessel functions of the first kind, by showing the monotonic decreasing property of the minimum values with respect to the order. In Section 3 our aim is improve the results stated in \cite{BariczPogany2014,BariczPonnusamySingh2016} and for this we show that the above mentioned results on Bessel functions of the first kind can be extended to general Bessel functions with a similar approach, however with much more attention on the zeros of cylinder functions. Section 4 is devoted to the study of regular Coulomb wave functions and we are going to prove here that some of the results on Bessel functions of the first kind can be extended also to regular Coulomb wave functions. The results of Section 4 complement naturally the known results in \cite{Baricz2015C}. Section 5 contains the corresponding results on Tur\'an type inequalities for modified Bessel functions of the second kind of purely imaginary order, a subject which has not been yet investigated in literature. In Section 6 we are interested on parabolic cylinder functions, and we show that the approach used in the previous section works well also in this case, and with this we complement the results stated in \cite{BariczIsmail2013} for parabolic cylinder functions. Section 7, which stands out from the rest, contains the study of the Tur\'anian of the Bessel functions of the first kind from the point of view of complex function theory, and it was motivated by the results of Section 2 as well as the recent researches in the literature concerning the inclusion properties of various special functions. Finally, in Section 8 we present a short discussion together with some questions and future research directions.

\section{\bf Tur\'an type inequalities for Bessel functions of the first kind}
\setcounter{equation}{0}

For $\mu>0$, let
\[\Delta_\mu(x)=J_\mu^2(x)-J_{\mu-1}(x)J_{\mu+1}(x)\]
be the Tur\'an determinant associated with the Bessel function of the first kind $J_{\mu}$, and consider the normalized quotient
\begin{equation}\label{eq:Phi-def}
 \Phi_\mu(x)=\frac{\Delta_\mu(x)}{J_\mu^2(x)}
 =1-\frac{J_{\mu-1}(x)J_{\mu+1}(x)}{J_\mu^2(x)}.
\end{equation}
Recall that for $x\notin\{j_{\mu,n}:n\ge1\}$, in view of Skovgaard's identity \cite{Skovgaard1954}
\begin{equation}\label{eq:ML}
\Phi_\mu(x)=\sum_{n\ge1}\frac{4j_{\mu,n}^2}{(x^2-j_{\mu,n}^2)^2},
\end{equation}
we clearly have that $\Phi_\mu(x)>0$ and
\begin{equation}\label{eq:Phi-second}
\Phi_\mu''(x)=\sum_{n\ge1} \frac{16j_{\mu,k}^2(5x^2+j_{\mu,n}^2)}{(x^2-j_{\mu,n}^2)^4}>0,
\end{equation}
where $j_{\mu,n}$ is the $n$th positive zero of the Bessel function $J_\mu.$ This in turn implies that $\Phi_\mu$ is strictly convex on every interval $(j_{\mu,n},j_{\mu,n+1}),$ and since $\Phi_\mu(x)\to+\infty$ at both endpoints, there is a unique point $\alpha_{\mu,n}\in(j_{\mu,n},j_{\mu,n+1})$ for which $\Phi_\mu'(\alpha_{\mu,n})=0$, and this point is the unique minimum of $\Phi_\mu$ on that interval. The strictly increasing property of the minimum values $\beta_{\mu,n}=\Phi_\mu(\alpha_{\mu,n})$ for $n\ge1$ was formulated as a conjecture by \'A. Baricz \cite[Conjecture 4.2]{Baricz2010} and restated by \'A. Baricz and T.K. Pog\'any \cite[Conjecture 2.3]{BariczPogany2014}. The latter paper also proves the crucial estimate $\beta_{\mu,n}<1$ for all $\mu>0$ and $n\ge1,$ see \cite[Section 2.8]{BariczPogany2014}. In this section our aim is to show that the above boundedness property of the sequence $\{\beta_{\mu,n}\}_{n\ge1}$ together with a Riccati identity, yields the conjectured monotonicity of the sequence $\left\{\beta_{\mu,n}\right\}_{n\geq1}.$ Moreover, we also show that the successive minimum value has an interesting property when the order $\mu$ varies. Figure \ref{Fig1} illustrates the result of the next theorem in the case when $\mu=2.$

\begin{theorem}\label{thm:main1}
If $\mu>0,$ then the sequence $\{\beta_{\mu,n}\}_{n\ge1}$ is strictly increasing and converges to $1$ from below. In other words, the sequence $\{\beta_{\mu,n}\}_{n\ge1}$ satisfies
\[ \frac1{\mu+1}<\beta_{\mu,1}<\beta_{\mu,2}<\ldots<\beta_{\mu,n}<\beta_{\mu,n+1}<\ldots<1\]
and consequently the following sharp Tur\'an-type inequalities are valid
$$J_\mu^2(x)-J_{\mu-1}(x)J_{\mu+1}(x)\geq\beta_{\mu,n}\cdot J_{\mu}^2(x)$$
for all $\mu>0$ and $x\in(j_{\mu,n},j_{\mu,n+1}).$ Moroever, for every fixed $n\ge1$, the function
$$\mu\mapsto\beta_{\mu,n}=\min_{j_{\mu,n}<x<j_{\mu,n+1}}\left(1-\frac{J_{\mu-1}(x)J_{\mu+1}(x)}{J_\mu^2(x)}\right)$$
is strictly decreasing on $(0,\infty)$.
\end{theorem}

\begin{figure}[t]
	\centering
	\includegraphics[width=0.9\textwidth]{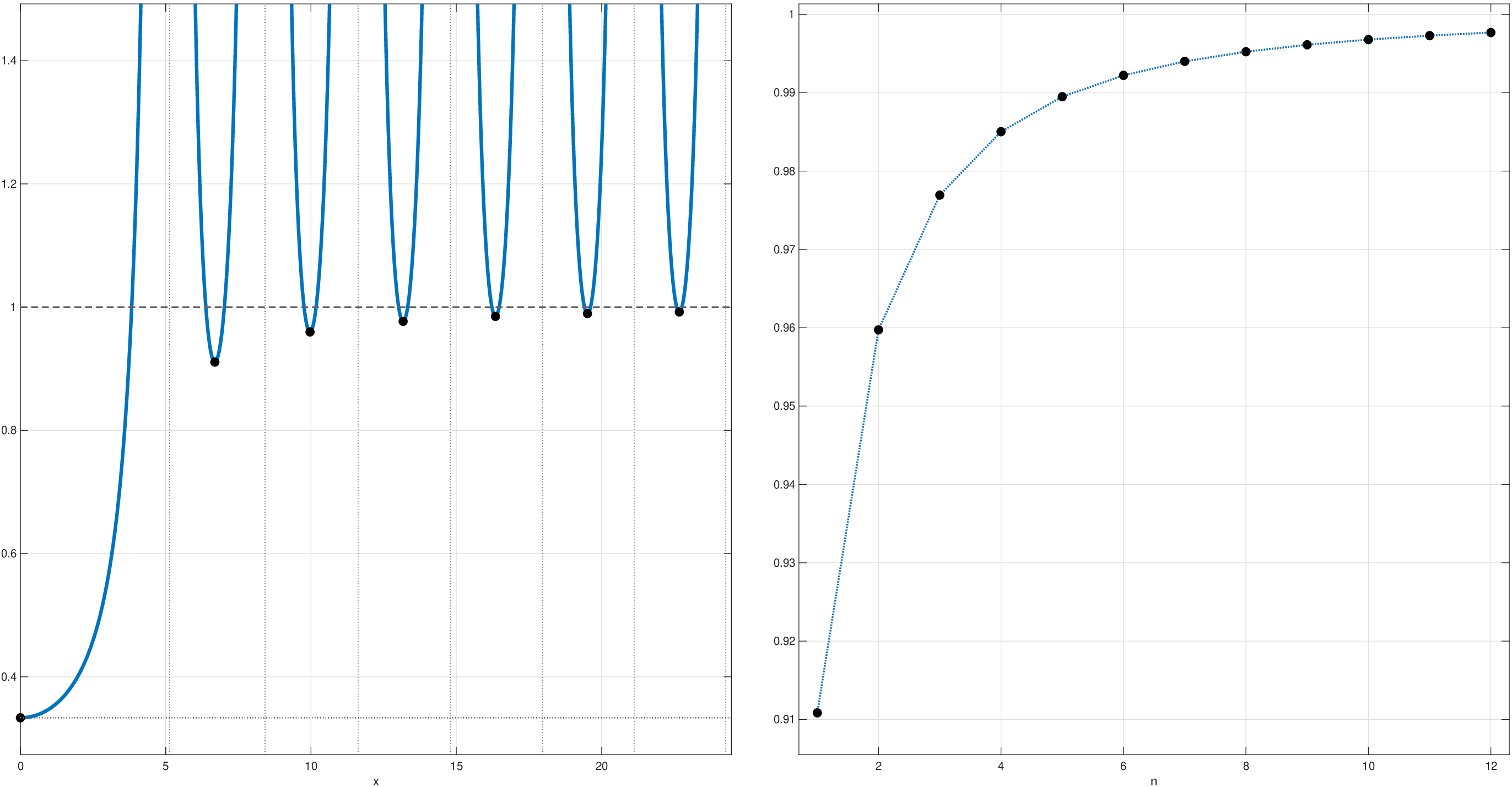}
	\caption{The graph of the first seven branches on the positive real axis of the normalized Tur\'anian $x\mapsto \Phi_{\mu}(x)$ for $\mu=2,$ together with the corresponding successive minimum value $\beta_{\mu,n}$ for $\mu=2$ and $n\in\{1,2,\ldots,12\}.$}
	\label{Fig1}
\end{figure}

\begin{proof}[\bf Proof of Theorem \ref{thm:main1}]
The standard recurrence relations \cite[eq. 10.6.2]{nist}
\[J_{\mu-1}(x)=J_\mu'(x)+\frac{\mu}{x}J_\mu(x), \qquad  J_{\mu+1}(x)=-J_\mu'(x)+\frac{\mu}{x}J_\mu(x)\]
clearly imply that
\[J_{\mu-1}(x)J_{\mu+1}(x)=\frac{\mu^2}{x^2}J_\mu^2(x)-\left[J_\mu'(x)\right)^2\]
and consequently
\begin{equation}\label{Phi:ydef}\Phi_\mu(x)=1+\frac{y_{\mu}^2(x)-\mu^2}{x^2},\quad \mbox{where}\quad y_{\mu}(x)=x\frac{J_\mu'(x)}{J_\mu(x)}.\end{equation}
On the other hand, in view of the Bessel differential equation \cite[eq. 10.2.1]{nist}
\[x^2J_\mu''(x)+xJ_\mu'(x)+(x^2-\mu^2)J_\mu(x)=0\]
we arrive at the Riccati differential equation
$$xy_{\mu}'(x)=\mu^2-x^2-y_{\mu}^2(x),$$
which in turn implies that $y_{\mu}'(x)=-x\Phi_\mu(x).$ In particular, because $\Phi_\mu(x)>0$, the function $y_{\mu}$ is strictly decreasing on every interval between two consecutive zeros of $J_\mu$. On the other hand, observe that combining the above relations we arrive at
\begin{equation}\label{Phi:der}\frac{1}{2}x\Phi_\mu'(x)=1-(1+y_{\mu}(x))\Phi_\mu(x).\end{equation}
We now use precisely the estimate supplied by \'A. Baricz and T.K. Pog\'any: in \cite[Section~2.8]{BariczPogany2014} they located the stationary point more sharply by
$\alpha_{\mu,n}\in (j_{\mu+1,n},j_{\mu-1,n+1}),$ and then used the monotonicity of $\Delta_\mu$ on this interval to obtain $\beta_{\mu,n}<1$ for all $\mu>0$ and $n\in\mathbb{N}.$ More explicitly, at $x=\alpha_{\mu,n}$ the derivative formula yields
\[ \Delta_\mu'(\alpha_{\mu,n})=\frac{2}{\alpha_{\mu,n}}J_{\mu-1}(\alpha_{\mu,n})J_{\mu+1}(\alpha_{\mu,n})>0,\]
and therefore
\[J_{\mu-1}(\alpha_{\mu,n})J_{\mu+1}(\alpha_{\mu,n})>0,\]
which is exactly $\beta_{\mu,n}<1,$ as we required. Since $\Phi_\mu'(\alpha_{\mu,n})=0$ and in view of \eqref{eq:ML} we have $\beta_{\mu,n}>0$ for all $\mu>0$ and $n\in\mathbb{N},$ it follows from \eqref{Phi:der} that the general term of the sequence $\left\{\tau_{\mu,n}\right\}_{n\geq1},$ defined by
$$\tau_{\mu,n}=y_{\mu}(\alpha_{\mu,n})=\alpha_{\mu,n}\frac{J_\mu'(\alpha_{\mu,n})}{J_\mu(\alpha_{\mu,n})},$$
satisfies $\beta_{\mu,n}\left(1+\tau_{\mu,n}\right)=1$ and $\tau_{\mu,n}>0$ for all $\mu>0$ and $n\in\mathbb{N}.$ Thus, evaluating $\Phi_\mu(x)$ in \eqref{Phi:ydef} at $x=\alpha_{\mu,n},$ we obtain
$$\alpha_{\mu,n}^2=\frac{(1+\tau_{\mu,n})(\mu^2-\tau_{\mu,n}^2)}{\tau_{\mu,n}},$$
which shows that $\tau_{\mu,n}<\mu$ for all $\mu>0$ and $n\in\mathbb{N}.$ Moreover, since the function
$$t\mapsto q_{\mu}(t)=\frac{(1+t)(\mu^2-t^2)}{t}$$
is strictly decreasing on $(0,\infty),$ the sequence $\left\{\alpha_{\mu,n}\right\}_{n\geq1}$ is increasing by definition, and $\alpha_{\mu,n}^2=q_{\mu}(\tau_{\mu,n}),$ we conclude that $\alpha_{\mu,n+1}^2=q_{\mu}(\tau_{\mu,n+1})>q_{\mu}(\tau_{\mu,n})=\alpha_{\mu,n}^2$
implies $\tau_{\mu,n+1}<\tau_{\mu,n}$ for all $\mu>0$ and $n\in\mathbb{N},$ or equivalently $\beta_{\mu,n+1}>\beta_{\mu,n}$ for all $\mu>0$ and $n\in\mathbb{N}.$ The above proof shows that $\left\{\tau_{\mu,n}\right\}_{n\geq1}$ is strictly decreasing and bounded below by zero, and hence $\tau_{\mu,n}\to \ell$ as $n\to\infty$ for some $\ell\ge0$. Since $\alpha_{\mu,n}>j_{\mu,n}$ and $j_{\mu,n}\to\infty$ as $n\to\infty$, we clearly have $\alpha_{\mu,n}\to\infty.$ However, if $\ell>0$, then by continuity of $q_\mu$ on $(0,\infty)$, we have $\alpha_{\mu,n}^2=q_\mu(\tau_{\mu,n})\to q_\mu(\ell)<\infty$ as $n\to\infty,$ which is a contradiction. Thus $\ell=0$, and $\beta_{\mu,n}\to 1$ as $n\to\infty.$

Finally, recall that Sz\'asz proved the sharp Tur\'an inequality \cite{Szasz1950}
$$J_\mu^2(x)-J_{\mu-1}(x)J_{\mu+1}(x)>\frac{1}{\mu+1}J_\mu^2(x),\quad \mu>0,$$
for the nonzero points relevant here. Hence $\beta_{\mu,n}>\frac1{\mu+1}$ for all $\mu>0$ and $n\in\mathbb{N}.$ Moreover, the small-argument expansion of $J_\mu$ shows that
$\Phi_\mu(x)\to\frac1{\mu+1}$ as $x\searrow0.$ Thus, if we define $\beta_{\mu,0}:=\frac1{\mu+1},$ then the complete statement is
$$\frac1{\mu+1}=\beta_{\mu,0}<\beta_{\mu,1}<\beta_{\mu,2}<\ldots<\beta_{\mu,n}<\beta_{\mu,n+1}<\ldots<1.$$

Now, we focus on the monotonicity of $\beta_{\mu,n}$ with respect to $\mu$ for a fixed $n\in\mathbb{N}.$ For this first we show that for $\mu>0$ and every $x>0$ such that $J_\mu(x)\ne0$, we arrive at
\begin{equation}\label{derformmu}\frac{\partial}{\partial\mu}\left(x\frac{J_\mu'(x)}{J_\mu(x)}\right)=\frac{2\mu}{J_\mu^2(x)}\int_0^x\frac{J_\mu^2(t)}{t}{\rm d}t.\end{equation}
Recall that the Bessel differential equation in self-adjoint form is
\[(xJ_{\mu}'(x))'+\left(x-\frac{\mu^2}{x}\right)J_{\mu}(x)=0.\]
Differentiating both sides of this equation with respect to $\mu$ we obtain
\[(x\omega_{\mu}'(x))'+\left(x-\frac{\mu^2}{x}\right)\omega_{\mu}(x)=\frac{2\mu}{x}J_{\mu}(x),\]
where $\omega_{\mu}(x)=\partial J_{\mu}(x)/\partial\mu.$ Now, we multiply the differentiated equation by $J_{\mu}(x)$, multiply the original
equation by $\omega_{\mu}(x)$, and we subtract the corresponding parts.  The potential terms cancel and we obtain the identity
\begin{equation}\label{eqLagr}\frac{{\rm d}}{{\rm d}x}\left(x(J_{\mu}(x)\omega_{\mu}'(x)-J_{\mu}'(x)\omega_{\mu}(x))\right)=\frac{2\mu}{x}J_{\mu}^2(x).\end{equation}
The expansion of $J_\mu(x)$ for small $x$ and its derivative with respect to the
order shows that the boundary term at the origin vanishes. Indeed, if we denote by $\epsilon_{\mu}^{-1}=2^{\mu}\Gamma(\mu+1)$ and $l(x)=\log(x/2)-\psi(\mu+1),$ where $\psi(t)=\Gamma'(t)/\Gamma(t)$ stands for the digamma function, then for small $x$ we have
$$J_{\mu}(x)=\epsilon_{\mu}x^{\mu}+\mathcal{O}(x^{\mu+2}),\quad J_{\mu}'(x)=\mu\epsilon_{\mu}x^{\mu-1}+\mathcal{O}(x^{\mu+1}),$$
$$\omega_{\mu}(x)=\epsilon_{\mu}x^{\mu}l(x)+\mathcal{O}(x^{\mu+2}|\log x|), \quad \omega_{\mu}'(x)=\epsilon_{\mu}x^{\mu}(\mu l(x)+1)+\mathcal{O}(x^{\mu+1}|\log x|)$$
and consequently as $x\to0$ we have
$$x(J_{\mu}(x)\omega_{\mu}'(x)-J_{\mu}'(x)\omega_{\mu}(x))=\frac{x^{2\mu}}{2^{2\mu}\Gamma^2(\mu+1)}(1+{o}(1)),$$
that is
$$\lim_{x\to0}\left(x(J_{\mu}(x)\omega_{\mu}'(x)-J_{\mu}'(x)\omega_{\mu}(x))\right)=0.$$
Integration from $0$ to $x$ of both sides of the equation \eqref{eqLagr}, where we originally changed $x$ to $t,$ therefore implies that
\[x\left(J_{\mu}(x)\omega_{\mu}'(x)-J_{\mu}'(x)\omega_{\mu}(x)\right)=2\mu\int_0^x\frac{J_{\mu}^2(t)}{t}{\rm d}t.\]
On the other hand, we clearly have
\[\frac{\partial y_\mu(x)}{\partial\mu}=x\frac{\omega_{\mu}'(x)J_{\mu}(x)-J_{\mu}'(x)\omega_{\mu}(x)}{J_{\mu}^2(x)},\]
which shows that the stated formula in \eqref{derformmu} is valid. We can now differentiate the minimum itself.  The point that deserves some
special attention is that the minimizer $\alpha_{\mu,n}$ also depends on $\mu$. Recall that at a minimum we have $\Phi_\mu'(\alpha_{\mu,n})=0.$ Moreover, differentiating the identity for $\Phi_\mu'$ and using again the relation $y_\mu'(x)=-x\Phi_\mu(x)$ we obtain $\Phi_\mu''(\alpha_{\mu,n})=2\beta_{\mu,n}^2>0.$ Hence the critical point is non-degenerate, and hence isolated, meaning no other critical point sits right next to it.  The implicit-function theorem therefore shows that $\alpha_{\mu,n}$ is locally a continuously differentiable function of $\mu$. Consequently, differentiating with respect to $\mu$ both sides of $\beta_{\mu,n}=\Phi_\mu(\alpha_{\mu,n})$ we arrive at
\[\frac{{\rm d}\beta_{\mu,n}}{{\rm d}\mu}=\frac{\partial\Phi_\mu}{\partial\mu}(\alpha_{\mu,n})+\Phi_\mu'(\alpha_{\mu,n})\frac{{\rm d}\alpha_{\mu,n}}{{\rm d}\mu}.\]
The second term vanishes because $\alpha_{\mu,n}$ is a critical point. Thus the motion of the minimizer does not contribute and by using the above results and notations we obtain
\begin{equation}\label{bracketpro}\frac{{\rm d}\beta_{\mu,n}}{{\rm d}\mu}=\frac{2}{\alpha_{\mu,n}^2}
\left(y_\mu(\alpha_{\mu,n})\frac{\partial y_\mu}{\partial\mu}(\alpha_{\mu,n})-\mu\right)=\frac{2\mu}{\alpha_{\mu,n}^2}
\left(\frac{2\tau_{\mu,n}}{J_{\mu}^2(\alpha_{\mu,n})}\int_0^{\alpha_{\mu,n}}\frac{J_\mu^2(t)}{t}{\rm d}t-1\right).\end{equation}
It remains to prove that the expression in brackets is negative. Now, let $u(r)=J_\mu(\alpha_{\mu,n} r),$ where $0<r<1.$ Since
$\tau_{\mu,n}=\alpha_{\mu,n}{J_{\mu}'(\alpha_{\mu,n})}/{J_\mu(\alpha_{\mu,n})},$ we have the Robin-type boundary condition $u'(1)=\tau_{\mu,n}\cdot u(1).$ On the other hand, the scaled Bessel differential equation reads $$-(ru'(r))'+\frac{\mu^2}{r}u(r)=\alpha_{\mu,n}^2\,ru(r).$$
Moreover, if we introduce the next integrals
\[\sigma_1=\int_0^1\frac{\left(u(r)\right)^2}{r}{\rm d}r\quad \mbox{and}\quad \sigma_2=\int_0^1r\left(u(r)\right)^2{\rm d}r,\]
then multiplying both sides of the above scaled Bessel differential equation by $u(r)$ and integrating by parts, we obtain
\[\alpha_{\mu,n}^2\sigma_2=\int_0^1r\left(u'(r)\right)^2{\rm d}r+\mu^2\sigma_1-\tau_{\mu,n}\left(u(1)\right)^2.\]
Thus, by completing the square
\[\int_0^1r\left(u'(r)-\frac{h}{r}u(r)\right)^2{\rm d}r=\alpha_{\mu,n}^2\sigma_2-(\mu^2-\tau_{\mu,n}^2)\sigma_1,\]
clearly the left-hand side is strictly positive and therefore for all $\mu>0$ and $r\in(0,1)$ we have
\[\sigma_1<\frac{\alpha_{\mu,n}^2}{\mu^2-\tau_{\mu,n}^2}\cdot\sigma_2=\frac{1+\tau_{\mu,n}}{\tau_{\mu,n}}\cdot\sigma_2.\]
On the other hand, by using the classical Lommel integral identity (see for example \cite[eq. 10.22.5]{nist} or \cite[p. 135]{Watson1944})
$$\int_0^xtJ_{\mu}^2(t){\rm d}t=\frac{x^2}{2}\left(J_{\mu}^2(x)-J_{\mu-1}(x)J_{\mu+1}(x)\right),$$
we arrive at
$$\sigma_2=\frac12\left(u(1)\right)^2\Phi_\mu(\alpha_{\mu,n})=\frac{\left(u(1)\right)^2}{2(1+\tau_{\mu,n})},$$
which in turn implies that $2\tau_{\mu,n}\sigma_1<\left(u(1)\right)^2$ or equivalently
$$\frac{2\tau_{\mu,n}}{J_{\mu}^2(\alpha_{\mu,n})}\int_0^{\alpha_{\mu,n}}\frac{J_\mu^2(t)}{t}{\rm d}t<1.$$
The expression in the brackets in \eqref{bracketpro} is therefore strictly negative, and with this we have shown that indeed the function $\mu\mapsto \beta_{\mu,n}$
is strictly decreasing on $(0,\infty)$ for $n\in\mathbb{N}$ fixed.
\end{proof}

\section{\bf Tur\'an type inequalities for general Bessel functions}
\setcounter{equation}{0}

For $\mu>0$ and $0\leq\alpha<\pi$, we consider the general Bessel function
\[C_\mu(x)=(\cos\alpha)J_\mu(x)-(\sin\alpha)Y_\mu(x),\]
where $Y_{\mu}$ stands for the Bessel function of the second kind. The functions $C_\mu$ satisfy the Bessel differential equation \cite[eq. 10.2.1]{nist}
$$ x^2C_\mu''(x)+xC_\mu'(x)+(x^2-\mu^2)C_\mu(x)=0$$
and the standard recurrence relations \cite[eq. 10.6.2]{nist}
$$C_\mu'(x)=C_{\mu-1}(x)-\frac{\mu}{x}C_\mu(x)=-C_{\mu+1}(x)+\frac{\mu}{x}C_\mu(x).$$
Let $c_{\mu,n}$ be the $n$th positive zero of $C_\mu$ and let us define the Tur\'anian
\[ \Delta_{\mu,\alpha}(x)=C_\mu^2(x)-C_{\mu-1}(x)C_{\mu+1}(x)\]
and its normalized quotient
$$\Phi_{\mu,\alpha}(x)=\frac{\Delta_{\mu,\alpha}(x)}{C_\mu^2(x)} =1-\frac{C_{\mu-1}(x)C_{\mu+1}(x)}{C_\mu^2(x)}, \qquad C_\mu(x)\ne0.$$
\'A. Baricz, S. Ponnusamy and S. Singh \cite{BariczPonnusamySingh2016} proved that, for $\mu>0$, $0<\alpha<\pi$, and $x\geq c_{\mu,1}$
$$\Delta_{\mu,\alpha}(x)>\frac1{\mu+1}C_\mu^2(x).$$
Consequently, on every interval $(c_{\mu,n},c_{\mu,n+1})$, $n\ge1$ we have
$$\Phi_{\mu,\alpha}(x)>\frac1{\mu+1}>0$$
and it is natural to ask whether the approach applied in the previous section for the Bessel functions of the first kind can be applied for the general Bessel functions. It turns out that the idea is working well also in the general case and the above positivity result is the only Tur\'an type inequality from the literature needed below. Observe that for $\alpha=0,$ Theorem \ref{thm:main2} reduces to Theorem \ref{thm:main1}, and in the case when $\alpha=\pi/2,$ Theorem \ref{thm:main2} may be written in terms of Bessel functions of the second kind $Y_{\mu},$ which complements and improves the results stated in \cite[Section 5.1]{BariczPogany2014}. It is worth to mention here that on a nodal interval whose left endpoint $c_{\mu,n}$ satisfies $c_{\mu,n}>\mu$, the unique minimum value of $\Phi_{\mu,\alpha}$ is strictly decreasing with respect to the order $\mu$. The proof, which is presented below, is a companion to the corresponding result for $J_\mu$ and the main difference is that a general cylinder function need not be regular at the origin, so the order-derivative identity must be integrated from the left nodal zero. The resulting boundary term is controlled by a classical order derivative formula for zeros of cylinder functions.

\begin{theorem}\label{thm:main2}
Let $\mu>0$, $0\leq\alpha<\pi.$ For every $n\ge1$, let $a_{\mu,n}$ be the unique critical point of $\Phi_{\mu,\alpha}(x)$ in $(c_{\mu,n},c_{\mu,n+1})$, and denote
$b_{\mu,n}=\Phi_{\mu,\alpha}(a_{\mu,n}).$ The sequence $\{b_{\mu,n}\}_{n\ge1}$ is strictly increasing and converges to $1$ from below. In other words, the sequence $\{b_{\mu,n}\}_{n\ge1}$ satisfies
\[ \frac1{\mu+1}<b_{\mu,1}<b_{\mu,2}<\ldots<b_{\mu,n}<b_{\mu,n+1}<\ldots<1\]
and consequently the following sharp Tur\'an-type inequalities are valid
$$C_\mu^2(x)-C_{\mu-1}(x)C_{\mu+1}(x)\geq b_{\mu,n}\cdot C_{\mu}^2(x)$$
for all $\mu>0,$ $0\leq\alpha<\pi$ and $x\in(c_{\mu,n},c_{\mu,n+1}).$ Moreover, for a fixed phase $\alpha\in[0,\pi)$ and a nodal index $n\geq1$ on every interval of orders
$\mu>0$ for which $c_{\mu,n}>\mu,$ the successive minimum
\[b_{\mu,n}=\min_{c_{\mu,n}<x<c_{\mu,n+1}}\left(1-\frac{C_{\mu-1}(x)C_{\mu+1}(x)}{C_\mu^2(x)}\right)\]
is a strictly decreasing function of $\mu$ on $(0,\infty)$.
\end{theorem}

\begin{remark}
{\em It is important to remark here that the condition $c_{\mu,n}>\mu$ is automatic in the case when $\alpha=0,$ that is $C_\mu=J_\mu$, and thus the above theorem recovers the previously proved monotonicity of $\beta_{\mu,n}$ with respect to the order. Moreover, the result on the monotonicity of the successive minimum it also applies to the case when $C_\mu=Y_\mu$. More generally, M.E. Muldoon and R. Spigler \cite{MuldoonSpigler1984} proved that, for $0\le\alpha\le{5\pi}/{6}$ the cylinder function $x\mapsto C_{\mu}(x)$ has no positive zero in $0<x\le\mu$.  Hence, in this phase range, $c_{\mu,n}>\mu$ for every positive zero and every $\mu>0$, and the above theorem says that ${\rm d}b_{\mu,n}/{\rm d}\mu<0$ for every nodal index $n$. For phases outside this range, a low positive zero may lie below the turning point and no assertion is made here for an exceptional low interval lying partly in $0<x\le\mu$.}
\end{remark}

For the general cylinder function the essential new ingredient in the proof of Theorem \ref{thm:main2} is a bound
for the variation of a positive zero with respect to the order. This result has been proved by L. Lorch \cite[p. 138]{Lorch1990}, however, for the sake of completeness, we present a slightly different proof below.

\begin{lemma}\label{lemmacyl}
Let $c_{\mu,n}$ be a positive zero of a fixed-phase $\alpha\in[0,\pi)$ cylinder function
$C_{\mu}(x)$ and fixed integer $n\geq1.$ If we suppose that $c_{\mu,n}>\mu>0,$ then $0<{\rm d}c_{\mu,n}/{\rm d}\mu<{c_{\mu,n}}/{\mu}.$ Equivalently, $\mu\mapsto c_{\mu,n}/\mu$ is strictly decreasing as long as the zero stays to the right of the turning point.
\end{lemma}

\begin{proof}[\bf Proof of Lemma \ref{lemmacyl}]
By using the classical formula \cite[eq. 10.21.17]{nist}
\[\frac{{\rm d}c_{\mu,n}}{{\rm d}\mu}=2c_{\mu,n}\int_0^\infty K_0(2c_{\mu,n}\sinh t)e^{-2\mu t}{\rm d}t\]
and the integral representation of the modified Bessel function of the second kind $K_0$
\[K_0(x)=\int_0^\infty e^{-x\cosh s}{\rm d}s, \quad x>0,\]
the elementary inequality $\sinh t>t$ for $t>0$ implies that
\[K_0(2c_{\mu,n}\sinh t)<K_0(2c_{\mu,n}t)=\int_0^\infty e^{-2(c_{\mu,n}\cdot\cosh s)t}{\rm d}s.\]
Thus, we arrive at
\[\frac{1}{c_{\mu,n}}\frac{{\rm d}c_{\mu,n}}{{\rm d}\mu}<2\int_0^{\infty}e^{-2\mu t}K_0(2c_{\mu,n}t){\rm d}t.\]
Now, since $c_{\mu,n}>\mu,$ the Laplace transform
$$\int_0^{\infty}e^{-at}K_0(bt){\rm d}t=\frac{\arccos(a/b)}{\sqrt{b^2-a^2}},\quad 0<a<b,$$
implies that
\[\frac{1}{c_{\mu,n}}\frac{{\rm d}c_{\mu,n}}{{\rm d}\mu}<\frac{\arccos(\mu/c_{\mu,n})}{\sqrt{c_{\mu,n}^2-\mu^2}}.\]
On the other hand, if $\alpha=\arccos(\mu/c_{\mu,n}),$ $\alpha\in(0,\pi/2),$ then we have $\tan\alpha>\alpha$ or equivalently
$$\alpha\cot\alpha=\frac{\mu\arccos(\mu/c_{\mu,n})}{\sqrt{c_{\mu,n}^2-\mu^2}}<1,$$
which completes the proof.
\end{proof}

\begin{proof}[\bf Proof of Theorem \ref{thm:main2}]
By using a similar approach as in the previous section we can easily show that
$$\Phi_{\mu,\alpha}(x)=1+\frac{y_{\mu,\alpha}^2(x)-\mu^2}{x^2},\quad \mbox{where}\quad y_{\mu,\alpha}(x)=x\frac{C_\mu'(x)}{C_\mu(x)},$$
satisfies $y_{\mu,\alpha}'(x)=-x\Phi_{\mu,\alpha}(x)$ as well as the relation
$$\frac{1}{2}x\Phi_{\mu,\alpha}'(x)=1-(1+y_{\mu}(x))\Phi_{\mu,\alpha}(x).$$
Thus, at a critical point $a>0$ of $\Phi_{\mu,\alpha}$ we therefore have $\Phi_{\mu,\alpha}(a)\left(1+y_{\mu,\alpha}(a)\right)=1$ and a second differentiation, evaluated at a critical point, gives the useful identity $\Phi_{\mu,\alpha}''(a)=2\Phi_{\mu,\alpha}^2(a)>0.$ Thus every critical point is a strict local minimum. At a zero $c=c_{\mu,n}$ of $C_\mu$, the recurrence relation
\[ C_{\mu-1}(c)+C_{\mu+1}(c)=\frac{2\mu}{c}C_\mu(c)=0\]
gives $C_{\mu-1}(c)=-C_{\mu+1}(c)$. Moreover, the simplicity of the zero implies $C_{\mu+1}(c)\ne0$. Hence $\Delta_\mu(c)=C_{\mu+1}^2(c)>0,$ and consequently
$\Phi_{\mu,\alpha}(x)\to+\infty$ as $x\to c_{\mu,n}^{+}$ or $x\to c_{\mu,n+1}^{-}.$ Therefore $\Phi_{\mu,\alpha}$ attains a minimum at some interior point of every interval $(c_{\mu,n},c_{\mu,n+1})$. Such a point is critical. Now, since every critical point is a strict local minimum, there cannot be two distinct critical points in the same interval: two distinct strict local minima would force an interior local maximum between them, hence another critical point at which the second derivative could not be strictly positive. Thus there is exactly one critical point, denoted by $a_{\mu,n}$.

Now, we are going to show that for every $\mu>0$, $0\leq\alpha<\pi$ and $n\ge1$ we have $b_{\mu,n}<1,$ where $b_{\mu,n}=\Phi_{\mu,\alpha}(a_{\mu,n}).$ Consider
$g_{\mu}(x)=x^{-\mu}C_\mu(x).$ Then, in view of \cite[eq. 10.6.6]{nist}, we clearly have $g_{\mu}'(x)=-x^{-\mu}C_{\mu+1}(x)$ and since $g_{\mu}(c_{\mu,n})=g_{\mu}(c_{\mu,n+1})=0,$ Rolle's theorem yields a point
$\xi_{\mu,n}\in(c_{\mu,n},c_{\mu,n+1})$ for which $g_{\mu}'(\xi_{\mu,n})=0$. Thus $C_{\mu+1}(\xi_{\mu,n})=0$ and $\Phi_{\mu,\alpha}(\xi_{\mu,n})=1.$ Since $a_{\mu,n}$ is the unique global minimum in the interval, we clearly have that $b_{\mu,n}\le1.$ Moreover, the inequality is strict. Indeed, at a zero of $C_{\mu+1}$ the corresponding recurrence relation implies that
\[y_{\mu,\alpha}(\xi_{\mu,n})=\xi_{\mu,n}\frac{C_\mu'(\xi_{\mu,n})}{C_\mu(\xi_{\mu,n})}=\mu.\]
and
\[\Phi_{\mu,\alpha}'(\xi_{\mu,n})=\frac2{\xi_{\mu,n}}[1-(1+\mu)]=-\frac{2\mu}{\xi_{\mu,n}}<0.\]
Thus $\xi_{\mu,n}$ is not the minimizer $a_{\mu,n}$, and hence $b_{\mu,n}<\Phi_{\mu,\alpha}(\xi_{\mu,n})=1.$

The rest of the proof of the monotonicity of the sequence $\{b_{\mu,n}\}_{n\geq1}$ go along the lines introduced in the proof of Theorem \ref{thm:main1}. Namely, if we consider the sequence $\varsigma_{\mu,n}=y_{\mu,\alpha}(a_{\mu,n}),$ then we clearly have that $b_{\mu,n}\left(1+\varsigma_{\mu,n}\right)=1$ and $0<\varsigma_{\mu,n}<\mu$ for all $\mu>0,$ $0\leq\alpha<\pi$ and $n\geq1.$ Moreover, since we have
\begin{equation}\label{eqamun}a_{\mu,n}^2=\frac{(1+\varsigma_{\mu,n})(\mu^2-\varsigma_{\mu,n}^2)}{\varsigma_{\mu,n}},\end{equation}
the monotonicity of the sequence $\{b_{\mu,n}\}_{n\ge1}$ follows from the monotonicity of the sequence $\{a_{\mu,n}\}_{n\ge1}$ and the monotonicity of the function $q_{\mu},$ defined in the proof of Theorem \ref{thm:main1}. Namely, $a_{\mu,n+1}^2=q_{\mu}(\varsigma_{\mu,n+1})>q_{\mu}(\varsigma_{\mu,n})=a_{\mu,n}^2$
implies $\varsigma_{\mu,n+1}<\varsigma_{\mu,n}$ for all $\mu>0,$ $0\leq\alpha<\pi$ and $n\in\mathbb{N},$ or equivalently $b_{\mu,n+1}>b_{\mu,n}$ for all $\mu>0,$ $0\leq\alpha<\pi$ and $n\in\mathbb{N}.$

Now, we focus on the monotonicity of $\mu\mapsto b_{\mu,n}$ for $\alpha\in[0,\pi)$ and $n\in\mathbb{N}$ fixed. Because the minimum is non-degenerate,
$a_{\mu,n}$ depends smoothly on $\mu$, locally in every parameter range where the same nodal branch is followed. Differentiating both sides of $b_{\mu,n}=\Phi_{\mu,\alpha}(a_{\mu,n})$ we obtain ${{\rm d}b_{\mu,n}}/{{\rm d}\mu}={\partial\Phi_{\mu,\alpha}}(a_{\mu,n})/{\partial\mu},$ because $\Phi_{\mu,\alpha}'(a_{\mu,n})=0.$ Moreover, at fixed $x$ we have
\[\frac{\partial\Phi_{\mu,\alpha}(x)}{\partial\mu}=\frac{2}{x^2}\left(y_{\mu,\alpha}(x)\frac{\partial y_{\mu,\alpha}(x)}{\partial\mu}-\mu\right)\]
and thus it is enough to show that, at the minimum
\[\varsigma_{\mu,n}\frac{\partial y_{\mu,\alpha}}{\partial\mu}(a_{\mu,n})<\mu.\]
Following the steps in the proof of Theorem \ref{thm:main1}, that is differentiating both sides of the Bessel differential equation with respect to $\mu$ and subtracting
the original equation multiplied by $\omega_{\mu,\alpha}(x)={\partial C_\mu(x)}/{\partial\mu},$ we arrive to the identity
\begin{equation}\label{Lident}\frac{{\rm d}}{{\rm d}x}\left(x(C_{\mu}(x)\omega_{\mu,\alpha}'(x)-C_{\mu}'(x)\omega_{\mu,\alpha}(x))\right)=\frac{2\mu}{x}C_{\mu}^2(x).\end{equation}
For $J_\nu$ we integrated this identity from the origin, where the boundary term vanished. This is the step which is no longer available for
a general cylinder function. Instead, we are going to integrate from the left zero $c_{\mu,n}$. Now, differentiating both sides of $C_\mu(c_{\mu,n})=0$ with respect to $\mu,$ we obtain $\omega_{\mu,\alpha}(c_{\mu,n})=-C_{\mu}'(c_{\mu,n}){{\rm d}c_{\mu,n}}/{{\rm d}\mu}.$ On the other hand, we clearly have
$$\frac{\partial y_{\mu,\alpha}}{\partial\mu}(x)=x\frac{C_{\mu}(x)\omega_{\mu,\alpha}'(x)-C_{\mu}'(x)\omega_{\mu,\alpha}(x)}{C_{\mu}^2(x)}$$
and thus, integration of both sides of the identity \eqref{Lident} from $c_{\mu,n}$ to $a_{\mu,n}$ therefore yields
\begin{equation}\label{eqcomplCPar}C_{\mu}^2(a_{\mu,n})\frac{\partial y_{\mu,\alpha}}{\partial\mu}(a_{\mu,n})=c_{\mu,n}\left(C_{\mu}'(c_{\mu,n})\right)^2\frac{{\rm d}c_{\mu,n}}{{\rm d}\mu}+
2\mu\int_{c_{\mu,n}}^{a_{\mu,n}}\frac{C_{\mu}^2(t)}{t}{\rm d}t.\end{equation}
Now, we focus on the integral on the right-hand side of the above relation. For this we consider the integrals
\[\pi_1=\int_{c_{\mu,n}}^{a_{\mu,n}}\frac{C_{\mu}^2(t)}{t}{\rm d}t\quad \mbox{and}\quad \pi_2=\int_{c_{\mu,n}}^{a_{\mu,n}}tC_{\mu}^2(t){\rm d}t.\]
Multiplication of the Bessel differential equation by $C_{\mu}$ and integration by parts, followed by completion of the square, implies that
\[0<\int_{c_{\mu,n}}^{a_{\mu,n}}t\left(C_{\mu}'(xt)-\frac{\varsigma_{\mu,n}}{t}C_{\mu}(t)\right)^2{\rm d}t=\pi_2-(\mu^2-\varsigma_{\mu,n}^2)\pi_1.\]
On the other hand, recall that the Lommel integral identity (see for example \cite[eq. 10.22.5]{nist} or \cite[p. 135]{Watson1944}) is valid for every cylinder function and thus
\[\int tC_\mu^2(t){\rm d}t=\frac{t^2}{2}\left(C_\mu^2(t)-C_{\mu-1}(t)C_{\mu+1}(t)\right).\]
Moreover, in view of the corresponding recurrence relations we have $C_{\mu-1}(c_{\mu,n})=C_{\mu}'(c_{\mu,n}),$ $C_{\mu+1}(c_{\mu,n})=-C_{\mu}'(c_{\mu,n})$ and at the minimum $a_{\mu,n}$ we arrive at
\[C_{\mu}^2(a_{\mu,n})-C_{\nu-1}(a_{\mu,n})C_{\nu+1}(a_{\mu,n})=b_{\mu,n}C_{\mu}^2(a_{\mu,n})\]
and consequently
\[\pi_2=\frac{a_{\mu,n}^2}{2}b_{\nu,n}C_{\mu}^2(a_{\mu,n})-\frac{c_{\mu,n}^2}{2}\left(C_{\mu}'(c_{\mu,n})\right)^2.\]
Now, by using the identity $b_{\mu,n}(1+\varsigma_{\mu,n})=1$ and the stationary relation \eqref{eqamun}, we obtain that
$$\pi_1<\frac{C_{\mu}^2(a_{\mu,n})}{2\varsigma_{\mu,n}}-\frac{c_{\mu,n}^2\left(C_{\mu}'(c_{\mu,n})\right)^2}{2(\mu^2-\varsigma_{\mu,n}^2)},$$
and combining this with the identity \eqref{eqcomplCPar} yields
\[\varsigma_{\mu,n}\frac{\partial y_\nu}{\partial\nu}(a_{\mu,n})
<\mu+\frac{\varsigma_{\mu,n}c_{\mu,n}\left(C_{\mu}'(c_{\mu,n})\right)^2}{C_{\mu}^2(a_{\mu,n})}
\left(\frac{{\rm d}c_{\mu,n}}{{\rm d}\mu}-\frac{\mu c_{\mu,n}}{\mu^2-\varsigma_{\mu,n}^2}\right).\]
Finally, in view of Lemma \ref{lemmacyl} we clearly have ${{\rm d}c_{\mu,n}}/{{\rm d}\mu}<{c_{\mu,n}}/{\mu}<\mu c_{\mu,n}/(\mu^2-\varsigma_{\mu,n}^2)$ and because $0<\varsigma_{\mu,n}<\mu,$ it follows that at the unique minimum
\[\frac{{\rm d}b_{\mu,n}}{{\rm d}\mu}=\frac{2}{a_{\mu,n}^2}\left(\varsigma_{\mu,n}\frac{\partial y_{\mu,\alpha}}{\partial\mu}(a_{\mu,n})-\mu\right)<0,\]
which completes the proof.
\end{proof}

\section{\bf Tur\'an type inequalities for regular Coulomb wave functions}
\setcounter{equation}{0}

The regular Coulomb wave function $F_L(\eta,\rho)$ is the standard solution of the Coulomb wave equation \cite[eq. 33.2.1]{nist}
\begin{equation}\label{eq:coulomb}
w''(\rho)+\left(1-\frac{2\eta}{\rho}-\frac{L(L+1)}{\rho^2}\right)w(\rho)=0,\qquad \rho>0,
\end{equation}
with regular behavior at the origin. We use the usual normalization for which
\[F_L(\eta,\rho)\sim C_L(\eta)\rho^{L+1}\]
as $\rho\searrow0,$ where $C_L(\eta)>0$. Tur\'an-type inequalities for regular Coulomb wave functions were developed by \'A. Baricz \cite{Baricz2015}, who also established a Mittag-Leffler expansion and several zero-interlacing properties. More recently, S.-Y. Chung \cite{Chung2026} obtained uniform bounds and further zero-separation and parameter-monotonicity results. In this section we are going to investigate a different object: not the standard order-shifted Coulomb Tur\'anian itself, but a Riccati functional naturally associated with the logarithmic derivative of $F_L$. The motivation comes naturally from the Bessel case. If $\eta=0$, then \cite[eq. 33.5.4]{nist}
\[F_L(0,\rho)=\sqrt{\frac{\pi\rho}{2}}\,J_{L+\frac{1}{2}}(\rho),\]
up to the standard normalization convention. Hence, after dividing by $\sqrt\rho$, the Coulomb equation becomes the Bessel-type differential equation of order $\nu=L+\frac12.$ The functional introduced below then reduces exactly to the normalized Tur\'an expression for Bessel functions of the first kind, and thus the theorem proved here is a direct deformation of the successive-minima result for the normalized Bessel Tur\'anian. To the best of our knowledge, the monotonicity of the successive interval minima considered below has not previously been recorded for regular Coulomb wave functions.

Throughout this section we assume that $L\ge0,$ $\eta>0,$ and $\nu=L+\frac12.$ We define $u(\rho)=\rho^{-1/2}F_L(\eta,\rho).$ The Coulomb differential equation \eqref{eq:coulomb} becomes
$$u''(\rho)+\frac1\rho u'(\rho)+\left(1-\frac{2\eta}{\rho}-\frac{\nu^2}{\rho^2}\right)u(\rho)=0.$$
Now, we define the logarithmic derivative
$$y_{L,\eta}(\rho)=\frac{\rho u'(\rho)}{u(\rho)}=\frac{\rho F_L'(\eta,\rho)}{F_L(\eta,\rho)}-\frac12$$
and a direct calculation yields the Riccati differential equation
$$\rho y_{L,\eta}'(\rho)=\nu^2-\rho^2+2\eta\rho-y_{L,\eta}^2(\rho).$$
We now introduce the expression
$$\Psi_{L,\eta}(\rho)=1-\frac{2\eta}{\rho}+\frac{y_{L,\eta}^2(\rho)-\nu^2}{\rho^2},$$
which clearly satisfies the basic identity $y_{L,\eta}'(\rho)=-\rho\Psi_{L,\eta}(\rho).$

\begin{theorem}\label{thm:main3}
Let $L\ge0,$ $\eta>0$ and let $0<x_{L,\eta,1}<x_{L,\eta,2}<\ldots$ denote the positive zeros of $F_L(\eta,\cdot)$. For every $n\ge1$, let $a_{L,\eta,n}$ be the unique critical point of $\Psi_{L,\eta}$ in $(x_{L,\eta,n},x_{L,\eta,n+1})$ and define
$b_{L,\eta,n}=\Psi_{L,\eta}(a_{L,\eta,n}).$ Then the sequence $\{b_{L,\eta,n}\}_{n\geq1}$ satisfies $$0<b_{L,\eta,1}<b_{L,\eta,2}<\ldots<b_{L,\eta,n}<b_{L,\eta,n+1}<\ldots< 1$$ and $\lim_{n\to\infty}b_{L,\eta,n}=1.$ Consequently, by using the notation $F_L=F_L(\eta,\rho)$ and  $$\Delta_{L,\eta}^C(\rho)=\left(L(L+1)+\eta^2\right)F_L^2-\sqrt{(L^2+\eta^2)((L+1)^2+\eta^2)}F_{L-1}F_{L+1}+\eta\left(\frac{F_L^2}{\rho}-F_LF_L'\right),$$
we arrive at the sharp Tur\'an type inequalities
$$\Delta_{L,\eta}^C(\rho)=L(L+1)\Psi_{L,\eta}(\rho)\left(F_L(\eta,\rho)\right)^2\ge L(L+1)b_{L,\eta,n}\cdot\left(F_L(\eta,\rho)\right)^2,$$
where $x_{L,\eta,n}<\rho<x_{L,\eta,n+1}$ and equality holds exactly at the points $\rho=a_{L,\eta,n}$. Moreover, the sharp constants $L(L+1)b_{L,\eta,n}$ increase to $L(L+1)$ as $n\to\infty.$
\end{theorem}

We mention that in the above theorem the correction term proportional to $\eta$ is intrinsic to the Coulomb
recurrences.  When $\eta=0$ it disappears, and consequently
\[\Psi_{L,0}(\rho)=\Phi_{L+\frac{1}{2}}(\rho)=1-\frac{J_{L-1/2}(\rho)J_{L+3/2}(\rho)}{J_{L+1/2}(\rho)^2},\]
so the construction reduces exactly to the normalized Bessel Tur\'an
quotient.

\begin{proof}[\bf Proof of Theorem \ref{thm:main3}]
Observe that
$$\Psi_{L,\eta}'(\rho)=\frac{2\eta}{\rho^2}+\frac{2y_{L,\eta}(\rho)y_{L,\eta}'(\rho)}{\rho^2}-\frac{2(y_{L,\eta}^2(\rho)-\nu^2)}{\rho^3}$$
and thus in view of the relations $y_{L,\eta}'(\rho)=-\rho\Psi_{L,\eta}(\rho)$ and $y_{L,\eta}^2(\rho)-\nu^2=\rho^2(\Psi_{L,\eta}(\rho)-1)+2\eta\rho$ we arrive at
$$\Psi_{L,\eta}'(\rho)=\frac{2Q(\rho)}{\rho},\quad \mbox{where}\quad Q(\rho)=1-\frac{\eta}{\rho}-(1+y_{L,\eta}(\rho))\Psi_{L,\eta}(\rho).$$
At a critical point $a>0$ we have $Q(a)=0$ and consequently
\[(1+y_{L,\eta}(a))\Psi_{L,\eta}(a)=1-\frac{\eta}{a}.\]
Moreover, differentiating both sides of $\Psi_{L,\eta}'(\rho)={2Q(\rho)}/{\rho}$ with respect to $\rho,$ we obtain that at a critical point $a>0$ we have
\[\Psi_{L,\eta}''(a)=\frac{2}{a}Q'(a)=\frac{2}{a}\left(\frac{\eta}{a^2}-y_{L,\eta}'(a)\Psi_{L,\eta}(a)\right)=2\Psi_{L,\eta}^2(a)+\frac{2\eta}{a^3}>0.\]
Hence every critical point is a strict local minimum, and there can be
only one on each interval.

It is well-known that the positive zeros of the regular Coulomb wave function are simple and form an unbounded sequence. Moreover, near a simple zero $x_{L,\eta,n}$ we clearly have that
\[\frac{F_L'(\eta,\rho)}{F_L(\eta,\rho)}\sim \frac{1}{\rho-x_{L,\eta,n}},\]
and consequently $y_{L,\eta}^2(\rho)/\rho^2\to+\infty$ and $\Psi_{L,\eta}(\rho)\to+\infty$ at both endpoints of each nodal interval. Thus, $\Psi_{L,\eta}$ attains at least one interior minimum on $(x_{L,\eta,n},x_{L,\eta,n+1}).$ Every critical point is a strict local minimum by the above strict convexity property. There cannot be two distinct critical points: two strict local minima would force an intervening critical point that is not a local minimum. Hence the critical point is unique and, because the function tends to $+\infty$ at both endpoints, it is the global minimum on $(x_{L,\eta,n},x_{L,\eta,n+1}).$

Let
\[\rho_t=\eta+\sqrt{\eta^2+L(L+1)}.\]
For $0<\rho<\rho_t$ the coefficient
\[c(\rho)=1-\frac{2\eta}{\rho}-\frac{L(L+1)}{\rho^2}\]
is negative. Since the regular solution is positive and increasing near
the origin, the differential equation shows that it cannot vanish before
$\rho_t$. If the first positive zero occurred before $\rho_t$, then $F_L(\eta,\rho)>0$ before that zero and the differential equation would imply
$F_L''(\eta,\rho)=-c(\rho)\cdot F_L(\eta,\rho)>0.$ Thus $F_L'(\eta,\rho)$ would remain positive and $F_L(\eta,\rho)$ could not reach zero, a contradiction. Hence $x_{L,\eta,1}>\rho_t$. Since $L\ge0$, we clearly have $\sqrt{\eta^2+L(L+1)}\ge\eta,$ which gives $\rho_t\ge2\eta$ and consequently $a_{L,\eta,n}>\eta$ for every $n$.

Moreover, $y(\rho)\to+\infty$ as $\rho\searrow x_{L,\eta,n}$.  If
$y_{L,\eta}(a_{L,\eta,n})\le0$, then at some point $\rho_0\le a_{L,\eta,n}$ one would have $y_{L,\eta}(\rho_0)=0$.
At such a point we would obtain
\[\Psi_{L,\eta}'(\rho_0)=\frac2\rho_0\left(\frac{\eta}{\rho_0}+\frac{\nu^2}{\rho_0^2}\right)>0,\]
whereas $\Psi_{L,\eta}'(\rho)<0$ to the left of its unique minimum, that is for every $\rho\in(x_{L,\eta,n},a_{L,\eta,n}).$ Hence we arrive at $y_{L,\eta,n}:=y_{L,\eta}(a_{L,\eta,n})>0.$ On the other hand, it follows immediately from the critical-point identity that
\[b_{L,\eta,n}=\frac{1-\eta/a_{L,\eta,n}}{1+y_{L,\eta,n}},\]
and therefore we have shown that $0<b_{L,\eta,n}<1$ for all $L,\eta>0$ and $n\in\mathbb{N}.$

Combining this formula with the definition of $\Psi_{L,\eta}$ yields $H(a_{L,\eta,n},y_{L,\eta,n})=0,$ where
\[H(a,y)=ya^2-\eta(2y+1)a+(1+y)(y^2-\nu^2).\]
For fixed $a>\eta$, this equation has exactly one positive root, denoted by $y(a)$. Indeed $H(a,0)<0$, while $H(a,y)\to+\infty$ as $y\to+\infty$,
and at every positive root
\[\frac{\partial H(a,y)}{\partial y}=\frac{a(a-\eta)+2y(1+y)^2}{1+y}>0.\]
Thus every positive zero is crossed with positive derivative and since $H(a,0)<0$, there can be only one positive crossing. Smoothness follows from the implicit-function theorem and the latter equation. Moreover, if we define
\[r(a)=\frac{a-\eta}{a(1+y(a))},\]
then implicit differentiation of $H(a,y(a))=0$ gives
\[y'(a)=-\frac{{\partial H(a,y(a))}/{\partial a}}{{\partial H(a,y(a))}/{\partial y}}=-\frac{2ay(a)-\eta(2y(a)+1)}{{\partial H(a,y(a))}/{\partial y}}\]
and
$$r'(a)=\frac{2y(a)\left(a(a-\eta)^2+\eta(1+y(a))^2\right)}{a^2(1+y(a))^2\cdot {\partial H(a,y(a))}/{\partial y}}>0.$$
Since the points $a_{L,\eta,n}$ lie in consecutive disjoint zero intervals, we have $a_{L,\eta,n+1}>a_{L,\eta,n}$ and because $b_{L,\eta,n}=r(a_{L,\eta,n})$, we obtain
$b_{L,\eta,n+1}>b_{L,\eta,n}$. Moreover, observe that $a_{L,\eta,n}\to\infty$ as $n\to\infty$ and the equation $H(a,y(a))=0$ forces
$y(a)\to0$ as $a\to\infty.$ Indeed, the equation $H(a,y(a))=0$ may be rewritten as $y(a)a(a-2\eta)=\eta a+(1+y(a))(\nu^2-y^2(a))$ and therefore
$y(a)a(a-2\eta)\leq \eta a+(1+y(a))\nu^2.$ Thus, for all sufficiently large $a$ we have
$$0<y(a)\leq\frac{\eta a+\nu^2}{a(a-2\eta)-\nu^2}$$
and since the right-hand side tends to zero, we clearly obtain $y(a_{L,\eta,n})\to0$ as $n\to\infty,$ which implies that
\[b_{L,\eta,n}=r(a_{L,\eta,n})=\frac{1-\eta/a_{L,\eta,n}}{1+y(a_{L,\eta,n})}\to1\]
as $n\to\infty.$

Finally, we mention that from the Coulomb differential equation we obtain
\[\Psi_{L,\eta}(\rho)=\frac{\left(F_L'(\eta,\rho)\right)^2-F_L(\eta,\rho)F_L''(\eta,\rho)-\rho^{-1}F_L(\eta,\rho)F_L'(\eta,\rho)}{F_L^2(\eta,\rho)},\]
which implies the validity of the corresponding expression in the theorem for $L(L+1)\Psi_{L,\eta}F_L^2(\eta,\rho).$ Here we used that in view of the notations  $$R_L=\sqrt{1+\frac{\eta^2}{L^2}}\quad \mbox{and}\quad S_L=\frac{L}{\rho}+\frac{\eta}{L}$$
the standard recurrence relations for the regular Coulomb wave functions are (see \cite[eq. 33.4.3]{nist} and \cite[eq. 33.4.4]{nist})
\[F_L'(\eta,\rho)=R_LF_{L-1}(\eta,\rho)-S_LF_L(\eta,\rho)\quad \mbox{and}\quad
F_L'(\eta,\rho)=S_{L+1}F_L(\eta,\rho)-R_{L+1}F_{L+1}(\eta,\rho).\]
\end{proof}

\section{\bf Tur\'an type inequalities for modified Bessel functions of purely imaginary order}
\setcounter{equation}{0}

For real order $\mu$, the classical Tur\'an type inequality for modified Bessel functions of the second kind is usually written in the reversed form
\[K_{\mu-1}(x)K_{\mu+1}(x)-K_\mu^2(x)>0,\quad x>0,\]
and the constant zero is best possible globally, see for example, \cite{Baricz2010}, \cite{Baricz2015}, \cite{BariczPonnusamy2013} or \cite{Segura2011}. The situation changes substantially for purely imaginary order. For fixed $\nu>0$, the function $K_{\mathrm{i}\nu}(x)$ is real for $x>0$, oscillates logarithmically near the origin, and possesses infinitely many positive zeros accumulating at zero, see \cite[Section 10.45]{nist}. The aim of this section is to show that the corresponding Tur\'an type inequality has a natural sharp formulation on each interval between consecutive zeros. In other words, the idea used above for Bessel functions, general Bessel functions and regular Coulomb wave functions can be also applied in the case of modified Bessel functions of the second kind of purely imaginary order.

Now, for $\nu>0$ let $f_{\nu}(x)=K_{\mathrm{i}\nu}(x).$ The function $f_{\nu}$ satisfies \cite[eq. 10.45.1]{nist}
\[x^2f_{\nu}''(x)+xf_{\nu}'(x)+(\nu^2-x^2)f_{\nu}(x)=0\]
and as $x\to0^+$ for $\gamma_{\nu}$ real and continuous with $\gamma_0=0,$ defined by \cite[eq. 10.24.3]{nist}, we have \cite[eq. 10.46.7]{nist}
\[
K_{\mathrm{i}\nu}(x)
=-\left(\frac{\pi}{\nu\sinh(\pi\nu)}\right)^{1/2}
\sin\!\left(\nu\ln\frac{x}{2}-\gamma_\nu\right)+\mathcal{O}(x^2),
\]
which implies that $f_{\nu}$ has infinitely many positive zeros tending to zero, see for example \cite{Laforgia1986}. Let us denote these zeros by $\kappa_{\nu,1}>\kappa_{\nu,2}>\ldots>\kappa_{\nu,n}>\ldots>0,$ and recall that all of these zeros are simple, by uniqueness for the differential equation. Recall also that for the modified Bessel function of the second kind we have the following well-known recurrence relations \cite[eq. 10.29.2]{nist}
\[K_{\mu-1}(x)=-K_\mu'(x)-\frac{\mu}{x}K_\mu(x),\quad K_{\mu+1}(x)=-K_\mu'(x)+\frac{\mu}{x}K_\mu(x)\]
and with $\mu=\mathrm{i}\nu$ we arrive at
\[K_{\mathrm{i}\nu-1}(x)K_{\mathrm{i}\nu+1}(x)=\left(f_{\nu}'(x)\right)^2+\frac{\nu^2}{x^2}f_{\nu}^2(x).\]
Since $K_{\mathrm{i}\nu-1}=K_{1-\mathrm{i}\nu}=\overline{K_{1+\mathrm{i}\nu}}$ for $x>0$, the product on the left-hand side is clearly $|K_{1+\mathrm{i}\nu}(x)|^2$. Moreover, if we define
\[T_\nu(x)=\frac{K_{\mathrm{i}\nu-1}(x)K_{\mathrm{i}\nu+1}(x)}{K_{\mathrm{i}\nu}^2(x)}-1=\left(\frac{f_{\nu}'(x)}{f_{\nu}(x)}\right)^2+\frac{\nu^2}{x^2}-1,\]
then with the notation $z_\nu(x)=x{f_{\nu}'(x)}/{f_{\nu}(x)},$ and in view of the modified Bessel differential equation, we arrive at $z_\nu'(x)=-xT_\nu(x)$ and
\begin{equation}\label{eqKim}-\frac{1}{2}xT_\nu'(x)=1+(1+z_\nu(x))T_\nu(x).\end{equation}

The next theorem complements the known results in the literature.

\begin{theorem}\label{thm:main4}
For every $\nu>0$ and $n\ge1$, the function $T_\nu$ has exactly one critical point $\alpha_{\nu,n}\in (\kappa_{\nu,n+1},\kappa_{\nu,n}),$ which is the unique minimum of $T_\nu$ on $(\kappa_{\nu,n+1},\kappa_{\nu,n}).$ If $\tau_{\nu,n}=T_\nu(\alpha_{\nu,n}),$ then $0<\tau_{\nu,1}<\tau_{\nu,2}<\ldots$ and $\tau_{\nu,n}\to\infty$ as $n\to\infty.$ Consequently, for all $\nu>0$ and $x\in(\kappa_{\nu,n+1},\kappa_{\nu,n})$ the following Tur\'an type inequality is valid
$$K_{\mathrm{i}\nu-1}(x)K_{\mathrm{i}\nu+1}(x)-K_{\mathrm{i}\nu}^2(x)\geq \tau_{\nu,n}\cdot K_{\mathrm{i}\nu}^2(x),$$
and the constant $\tau_{\nu,n}$ is best possible.
\end{theorem}

\begin{remark}
{\em We would like to mention here that since $K_{\mathrm{i}\nu-1}(x)K_{\mathrm{i}\nu+1}(x)=|K_{1+\mathrm{i}\nu}(x)|^2,$ the above Tur\'an type inequality may also be written as
\[|K_{1+\mathrm{i}\nu}(x)|^2-K_{\mathrm{i}\nu}^2(x)\ge \tau_{\nu,n}\cdot K_{\mathrm{i}\nu}^2(x).\]
Thus, if we use the conventional determinant notation
\[\Delta_\nu(x)=\left(K_{\mathrm{i}\nu}(x)\right)^2-K_{\mathrm{i}\nu-1}(x)K_{\mathrm{i}\nu+1}(x),\]
then the sharp statement becomes
\[\Delta_\nu(x)\le -\tau_{\nu,n}K_{\mathrm{i}\nu}^2(x),\]
and now the best constants in the conventional normalization are negative and strictly decreasing to $-\infty$ as the intervals approach the origin.

Moreover, it is worth to mention that the asymptotic formula for $K_{\mathrm{i}\nu}(x)$ for small argument, mentioned above, shows that the zeros are asymptotically geometric, that is
${\kappa_{\nu,n+1}}/{\kappa_{\nu,n}}\to e^{-\pi/\nu}$ as $n\to \infty$ and since the critical-point condition gives $z_{\nu}(\alpha_{\nu,n})\to-1$ we further obtain
${\alpha_{\nu,n+1}}/{\alpha_{\nu,n}}\to e^{-\pi/\nu}$ as $n\to\infty$ and therefore ${\tau_{\nu,n+1}}/{\tau_{\nu,n}}\to e^{2\pi/\nu},$ as $n\to \infty.$ Hence the imaginary-order Macdonald case is qualitatively different from the Bessel, general Bessel or regular Coulomb cases, that is $\tau_{\nu,n}\to\infty$ geometrically, while $\alpha_{\nu,n}^2\tau_{\nu,n}\to1+\nu^2$ as $n\to\infty.$ The difference is explained by the zero geometry: the relevant zeros of $J_\mu$ and of the regular Coulomb wave function escape to $+\infty$, whereas those of $K_{\mathrm{i}\nu}$ accumulate at the singular endpoint $0$.}
\end{remark}

\begin{proof}[\bf Proof of Theorem \ref{thm:main4}]
Since the zeros of $f_{\nu}$ are simple, we have $T_\nu(x)\to+\infty$ as $x$ approaches either endpoint of the interval $(\kappa_{\nu,n+1},\kappa_{\nu,n}).$ Hence $T_\nu$ has at least one interior minimum. Let $\alpha$ be any critical point and write $\tau=T_\nu(\alpha)$. Thus if $T_\nu'(\alpha)=0,$ the relation \eqref{eqKim} clearly implies $1+(1+z_\nu(\alpha))\tau=0$ and differentiating the formula for $T_\nu'$ and using $z_\nu'(x)=-xT_\nu(x)$ gives, at a critical point, the following $T_\nu''(\alpha)=2\tau^2>0.$
Thus every critical point is a strict local minimum. Two such critical points in the same interval would force an intervening critical point, which is not a local minimum. Hence the critical point is unique. Now, at the critical point we have $z_\nu(\alpha)=-1-1/\tau$ and substitution into
\[T_\nu(x)+1=\frac{z_\nu^2(x)+\nu^2}{x^2}\]
implies
$$\alpha_{\nu,n}^2=\frac{\tau_{\nu,n}+1}{\tau_{\nu,n}^2}+\frac{\nu^2}{\tau_{\nu,n}+1}.$$
On the other hand, the function
\[t\mapsto h_\nu(t)=\frac{t+1}{t^2}+\frac{\nu^2}{t+1}\]
is strictly decreasing on $(0,\infty)$ and this in turn implies that clearly $\tau_{\nu,1}<\tau_{\nu,2}<\ldots,$ since $\alpha_{\nu,1}>\alpha_{\nu,2}>\ldots>0$ and $\alpha_{\nu,n}^2=h_{\nu}(\tau_{\nu,n}).$ Finally $\alpha_{\nu,n}\to0$, while $h_\nu$ decreases from $+\infty$ to $0$ on $(0,\infty)$. Therefore $\tau_{\nu,n}\to\infty$.

Now, we fix the nodal index $n$ and regard $\tau_{\nu,n}$ as a function of $\nu.$ Since the critical point is non-degenerate, the implicit-function theorem
gives local smooth dependence of $\alpha_{\nu,n}$ on $\nu,$ and we obtain ${{\rm d}\tau_{\nu,n}}/{{\rm d}\nu}={\partial T_\nu(\alpha_{\nu,n})}/{\partial\nu}.$ On the other hand we have
at fixed $x$ the next formula
\[\frac{\partial T_\nu(x)}{\partial\nu}=\frac{2}{x^2}\left(z_\nu(x)\frac{\partial z_\nu(x)}{\partial\nu}+\nu\right).\]
In view of the above formula we are going to derive an integral representation for the order derivative of $z_\nu.$ For this let $v(x)={\partial f_\nu(x)}/{\partial\nu}$ and observe that the self-adjoint modified Bessel differential equation
\[(xf_\nu'(x))'+\left(\frac{\nu^2}{x}-x\right)f_\nu(x)=0\]
implies after differentiation with respect to $\nu$ the following
\[(xv'(x))'+\left(\frac{\nu^2}{x}-x\right)v(x)=-\frac{2\nu}{x}f_\nu(x).\]
Multiplying both sides of the latter equation by $f_{\nu}(x)$ and subtracting the corresponding parts of the original equation multiplied by $v(x)$ yields
\begin{equation}\label{eqmodKLag}\frac{{\rm d}}{{\rm d}x}\left(x(f_\nu(x)v'(x)-f_\nu'(x)v(x))\right)=-\frac{2\nu}{x}f_\nu^2(x),\end{equation}
very similarly as in the case of Bessel functions of the first kind and general Bessel functions. Now, recall that for fixed order $\nu$ the next asymptotics is valid
for large $x$ (see \cite[eq. 10.40.2]{nist})
$$K_{\nu}(x)\sim\sqrt{\frac{\pi}{2x}}e^{-x}\left(1+\frac{4\nu^2-1}{8x}+\mathcal{O}\left(x^{-2}\right)\right)$$
and thus, we clearly have that
$$K_{\mathrm{i}\nu}(x)\sim\sqrt{\frac{\pi}{2x}}e^{-x}\left(1-\frac{4\nu^2+1}{8x}+\mathcal{O}\left(x^{-2}\right)\right),$$
as $x\to\infty.$ Thus, as $x\to\infty$ we obtain that $f_{\nu}(x)=\mathcal{O}\left(e^{-x}x^{-1/2}\right),$ $f_{\nu}'(x)=\mathcal{O}\left(e^{-x}x^{-1/2}\right),$ $v(x)=\mathcal{O}\left(e^{-x}x^{-3/2}\right)$ and $v(x)=\mathcal{O}\left(e^{-x}x^{-3/2}\right),$ which in turn implies that
$$x\left(f_\nu(x)v'(x)-f_\nu'(x)v(x)\right)=\mathcal{O}\left(e^{-2x}x^{-1}\right)$$
as $x\to \infty.$ Consequently, because both $K_{\mathrm{i\nu}}(x)$ and its order derivative decay exponentially as $x\to\infty,$ integration from $x$ to $\infty$ of both sides of the equation \eqref{eqmodKLag} (where we tacitly changed $x$ to $t$) implies that
\[\frac{\partial z_\nu}{\partial\nu}(x)=\frac{2\nu}{f_\nu^2(x)}\int_x^\infty\frac{f_\nu^2(t)}{t}{\rm d}t\]
and thus it remains to prove that
$$\varpi_1=\int_{\alpha_{\nu,n}}^\infty\frac{f_\nu^2(t)}{t}{\rm d}t>\frac{f_{\nu}^2(\alpha_{\nu,n})}{2q_{\nu,n}},$$
where $q_{\nu,n}=-z_{\nu}(\alpha_{\nu,n})=1+1/\tau_{\nu,n}.$ Thus, we consider the positive quadratic form
\[\varpi=\int_{\alpha_{\nu,n}}^\infty t\left(f_\nu'(t)+\frac{q_{\nu,n}}{t}f_\nu(t)\right)^2{\rm d}t=(\nu^2+q_{\nu,n}^2)\varpi_1-\varpi_2,\quad \mbox{where}\quad \varpi_2=\int_{\alpha_{\nu,n}}^\infty t f_\nu^2(t){\rm d}t\]
and in the evaluation of this quadratic form we used the self-adjoint form of the modified Bessel differential equation together with integration by parts. On the other hand, in view of the following identity
\[\frac{{\rm d}}{{\rm d}x}\left(x^2\left(f_\nu'(x)\right)^2+(\nu^2-x^2)f_\nu^2(x)\right)=-2xf_\nu^2(x),\]
which can be deduced with the help of the modified Bessel differential equation, we obtain that
$$2\varpi_2=\alpha_{\nu,n}^2\left(f_\nu'(\alpha_{\nu,n})\right)^2+(\nu^2-\alpha_{\nu,n}^2)f_{\nu}^2(\alpha_{\nu,n})=
f_{\nu}^2(\alpha_{\nu,n})\left(q_{\nu,n}^2+\nu^2-\alpha_{\nu,n}^2\right)=\alpha_{\nu,n}^2\tau_{\nu,n}f_{\nu}^2(\alpha_{\nu,n}).$$
Consequently, we have
\[\varpi=\alpha_{\nu,n}^2(1+\tau_{\nu,n})\varpi_1-\frac{\alpha_{\nu,n}^2\tau_{\nu,n}}{2}f_{\nu}^2(\alpha_{\nu,n})>0,\]
and with this the proof is complete.
\end{proof}

\section{\bf Tur\'an type inequalities for parabolic cylinder functions}
\setcounter{equation}{0}

Before we discuss the case of Tur\'an type inequalities for parabolic cylinder functions, first we recall some known results in the literature. Note that \'A. Baricz and M.E.H. Ismail \cite{BariczIsmail2013} obtained sharp two-sided Tur\'an type inequalities for parabolic cylinder functions in a non-oscillatory negative-order range, by using integral representations for quotients of parabolic cylinder functions, and J. Segura \cite{Segura2012} derived related bounds by a general method for first-order difference-differential systems. Moreover, it is worth to mention here that in a stochastic-control setting, D. Becherer, D. Bilarev and P. Frentrup \cite{BechererBilarevFrentrup2018} conjectured the strict monotonicity of a quotient of Hermite functions. T. Koch \cite{Koch2020} proved this conjecture, obtained sharp universal bounds, and translated the result to parabolic cylinder functions from which the Tur\'an type inequalities follow as a consequence. T. Koch's proof is probabilistic and uses the interpretation of Hermite functions as eigenfunctions associated with an Ornstein–Uhlenbeck process. The above mentioned results concern the global behavior in a non-oscillatory range. The problem considered in this section complements the above mentioned results since for positive order, real zeros produce poles of the normalized Tur\'an expression, and the natural objects are its minima on the intervals between consecutive positive zeros. The positivity of the corresponding Tur\'anian of parabolic cylinder functions for positive order has been claimed in \cite[Remark 1]{BariczIsmail2013}, and now we are going to prove this in the present section. However, it is worth to mention here that the minimum values does not have the increasing property as in the case of Bessel, general Bessel or regular Coulomb wave functions.

For parabolic cylinder functions we use the notation \cite[12.2.5]{nist}
\[ D_\nu(x)=U\left(-\nu-\frac12,x\right).\]
The real-zero classification for $U(a,x)$ in \cite[Section 12.11]{nist} states that if $-2n-\frac32<a<-2n+\frac12$ for $n\in\mathbb{N},$ then $U(a,x)$ has $n$ positive real zeros, or equivalently, if $2n-1<\nu<2n+1$ for $n\in\mathbb{N},$ then $D_{\nu}(x)$ has $n$ positive real zeros. Thus, if $N_+(\nu)$ counts the number of strictly positive zeros of $D_\nu,$ then for $0<\nu\leq 1$ we have $N_+(\nu)=0,$ and in the case when $2n-1<\nu<2n+1$ we have $N_+(\nu)=n.$ At the right endpoint $\nu=2n+1$, the connection between parabolic cylinder functions and Hermite polynomials shows that $x=0$ is a zero, in addition to the $n$ strictly positive zeros. Thus the origin is not included in $N_+(\nu)$. Consequently, intervals between consecutive strictly positive zeros exist exactly when $\nu>3$, and if $n=N_+(\nu)$, their number is $n-1$. Moreover, when $\nu$ is
an odd integer at least $3$, the interval from the zero at the origin to the first positive zero is an additional pole-to-pole interval, which will be discussed
separately in Remark \ref{rem:end-branches}.

Now, for $\nu>0$ we define
\[\Pi_\nu(x)=1-\frac{D_{\nu-1}(x)D_{\nu+1}(x)}{D_\nu^2(x)},\]
where $D_\nu$ is the parabolic cylinder function and when $\nu>1$, let $0<d_{\nu,1}<\ldots<d_{\nu,n}$ be its strictly positive zeros, where
$N=N_+(\nu).$

The next theorem is the corresponding result on the Tur\'anian of parabolic cylinder functions. Figure \ref{Fig2} illustrates the result of this theorem.

\begin{theorem}\label{thm:main5}
Let $\nu>3$. For every $n\in\{1,\ldots,N-1\}$, the function $\Pi_\nu$ has exactly one critical point $\gamma_{\nu,n}\in(d_{\nu,n},d_{\nu,n+1}),$ which is a strict minimum. If $\delta_{\nu,n}=\Pi_\nu(\gamma_{\nu,n}),$ then we have $1>\delta_{\nu,1}>\delta_{\nu,2}>\ldots>\delta_{\nu,N-1}>0.$ Consequently, for all $x\in(d_{\nu,n},d_{\nu,n+1})$ we have the next Tur\'an type inequality
\[ D_\nu^2(x)-D_{\nu-1}(x)D_{\nu+1}(x)\ge \delta_{\nu,n}D_\nu^2(x),\]
where $\delta_{\nu,n}$ is the best possible constant.
\end{theorem}

\begin{figure}[t]
	\centering
	\includegraphics[width=0.9\textwidth]{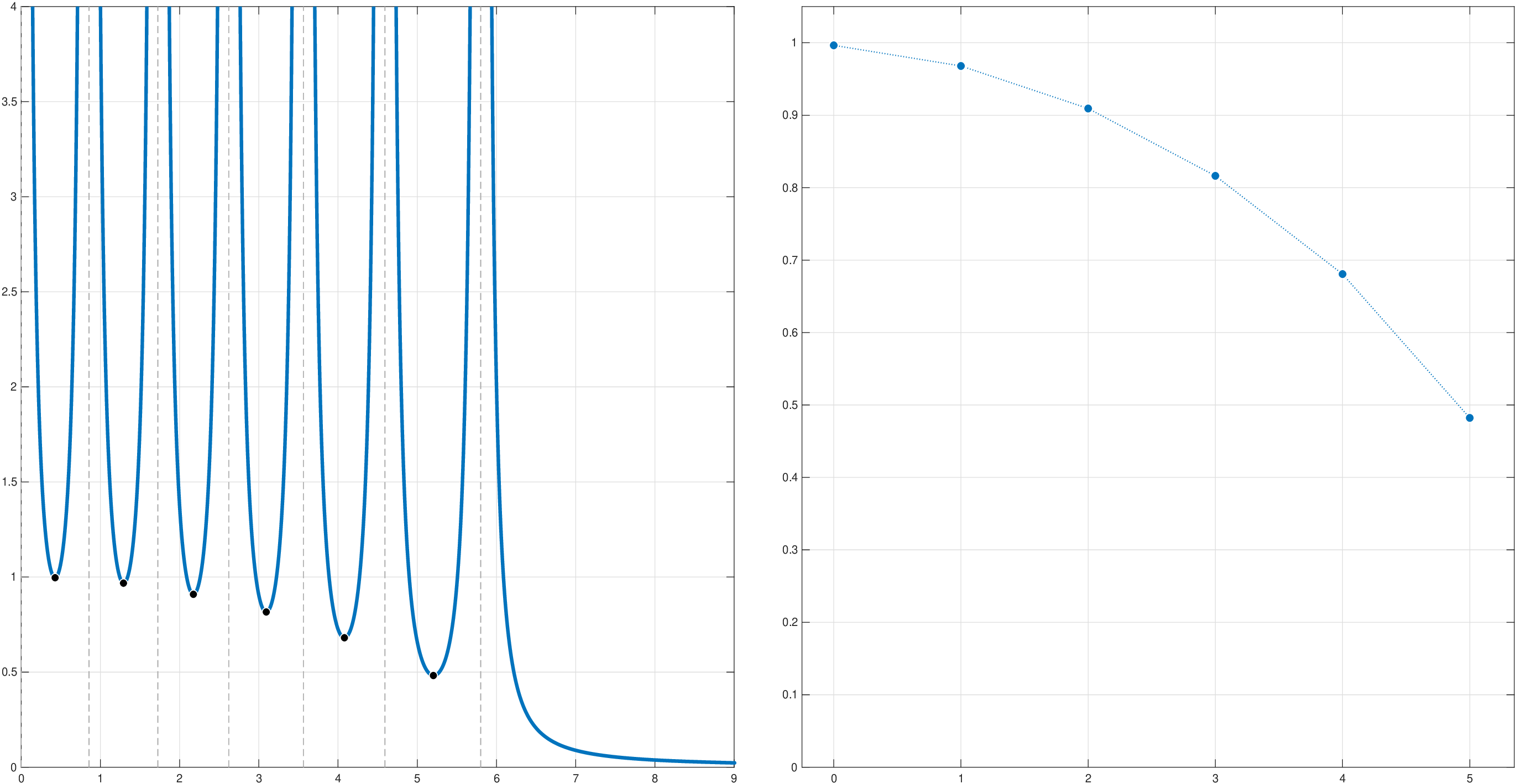}
	\caption{The graph of the seven branches on the positive real axis of the normalized Tur\'anian $x\mapsto \Pi_{\mu}(x)$ for $\mu=13,$ together with the corresponding successive minimum value $\delta_{\mu,n}$ for $\mu=13$ and $n\in\{1,2,\ldots,6\}.$}
	\label{Fig2}
\end{figure}

\begin{proof}[\bf Proof of Theorem \ref{thm:main5}]
Let us consider the expressions
\[r_{\nu}(x)=\frac{D_\nu'(x)}{D_\nu(x)}\quad \mbox{and}\quad \Omega_\nu(x)=\Pi_\nu(x)+\frac{1}{2\nu}.\]
The standard recurrence relations for parabolic cylinder functions are available at \cite[eq. 12.8.2]{nist} and \cite[eq. 12.8.3]{nist}, and they can be rewritten in terms of $D_\nu$ as
$$D_{\nu}'(x)=\nu D_{\nu-1}(x)-\frac{1}{2}xD_{\nu}(x)=\frac{1}{2}xD_{\nu}(x)-D_{\nu+1}(x),$$
and these in turn imply that
\[\Pi_\nu(x)=1-\frac{x^2}{4\nu}+\frac{r_{\nu}^2(x)}{\nu}.\]
By using the Weber differential equation \cite[eq. 12.2.4]{nist}
\[D_\nu''(x)+\left(\nu+\frac12-\frac{x^2}{4}\right)D_\nu(x)=0\]
we obtain
\[r_{\nu}'(x)=-\nu\Omega_\nu(x)\quad \mbox{and}\quad \Omega_\nu'(x)=-2r_{\nu}(x)\Omega_\nu(x)-\frac{x}{2\nu}.\]
Thus $\Pi_\nu$ and $\Omega_\nu$ have the same critical points. Now, let $a>0$ be such a critical point and set $s=\Omega_\nu(a)$. The critical-point equation yields
$r_{\nu}(a)=-{a}/{(4\nu s)}$ and substitution into the preceding representation implies that
\[a^2=\zeta_\nu(s),\quad \mbox{where}\quad \zeta_\nu(s)=\frac{16\nu^2s^2\left(\nu+\frac12-\nu s\right)}{4\nu^2s^2-1}.\]
Moreover, a second differentiation at the critical point gives
\[\Pi_\nu''(a)=\frac{4\nu^2s^2-1}{2\nu}.\]

Now, we are going to show that all positive zeros of $D_\nu$ lie to the left of the turning point
$x_{\mathrm t}=2\sqrt{\nu+\tfrac12}$. For $x>x_{\mathrm t}$ the Weber differential equation becomes
\[D_\nu''(x)=l_\nu(x)D_\nu(x),\quad \mbox{where}\quad l_\nu(x)=\frac{x^2}{4}-\nu-\frac12>0.\]
Suppose that $D_\nu$ has a zero in $[x_{\mathrm t},\infty).$ Since as $x\to+\infty$ we have \cite[12.9.1]{nist} $D_\nu(x)\sim x^\nu e^{-x^2/4}>0,$ there is a last such zero, which we denote by $z.$ Then $D_\nu(x)>0$ for $x>z$ and $D_\nu(z)=0.$ Moreover, we have that $D_\nu(x)\to0$ as $x\to\infty,$ which implies that $D_\nu$ attains a positive maximum at some point $c\in(z,\infty).$ At this point we have $D_\nu'(c)=0,$ and $D_\nu''(c)\le0.$ On the other hand, since $c>z\ge x_{\mathrm t},$ we obtain $D_\nu''(c)=l_\nu(c)D_\nu(c)>0,$ because both $l_\nu(c)$ and $D_\nu(c)$ are positive. This is a contradiction. Therefore every positive zero of $D_\nu$ lies strictly to the left of the turning point, that is we have $d_{\nu,j}<2\sqrt{\nu+\frac12}.$ Consequently, if $a\in(d_{\nu,n},d_{\nu,n+1}),$ then $0<a<d_{\nu,n+1}<2\sqrt{\nu+\frac12}$ and therefore every critical point between consecutive positive zeros satisfies $a^2<4\nu+2.$

Now, suppose that $s<0$. Since $a^2=\zeta_\nu(s)>0$, we necessarily have $2\nu s<-1$ and writing
$u=-2\nu s>1$ we arrive at
\[\zeta_\nu(s)=\frac{2u^2(u+2\nu+1)}{u^2-1}>4\nu+2,\]
contrary to $a^2<4\nu+2$. Hence $s>0$ and the condition $\zeta_\nu(s)>0$ then clearly implies
\[\frac{1}{2\nu}<s<1+\frac{1}{2\nu}.\]
Therefore we arrive at
\[0<\Pi_\nu(a)=s-\frac{1}{2\nu}<1\quad \mbox{and}\quad \Pi_\nu''(a)>0\]
and thus every critical point between consecutive positive zeros is a strict local minimum. At each endpoint of $(d_{\nu,n},d_{\nu,n+1})$ the logarithmic derivative $r_{\nu}$ has a simple pole, and hence $\Pi_\nu(x)\to+\infty$. Therefore at least one interior critical point exists. Since every critical point is a strict minimum, there cannot be two of them in the same interval: two distinct strict minima would force an intervening critical point which is not a minimum. This in turn implies that the critical point $\gamma_{\nu,n}$ is unique and $0<\delta_{\nu,n}<1$.

It remains to compare the minimum values.  At $a=\gamma_{\nu,n}$ we write $s_{\nu,n}=\delta_{\nu,n}+{1}/{(2\nu)}.$ Then $\gamma_{\nu,n}^2=\zeta_\nu(s_n)$ and on the admissible interval $1/(2\nu)<s<1+1/(2\nu)$ we obtain that
\[\zeta_\nu'(s)=-\frac{16\nu^2s\left(4\nu^3s^3-3\nu s+2\nu+1\right)}{(4\nu^2s^2-1)^2}<0.\]
Indeed, with $t=2\nu s>1$, the factor in parentheses equals $\tfrac12(t^3-3t+4\nu+2)$, which is positive because $t^3-3t+2=(t-1)(t^2+t-2)>0$. Since
$\gamma_{\nu,1}<\gamma_{\nu,2}<\ldots<\gamma_{\nu,N-1},$ and $\zeta_\nu$ is strictly decreasing on the admissible branch, it follows that
$s_{\nu,1}>s_{\nu,2}>\ldots>s_{\nu,N-1}$. Subtracting $1/(2\nu)$ gives the asserted strictly decreasing property $1>\delta_{\nu,1}>\delta_{\nu,2}>\ldots>\delta_{\nu,N-1}>0.$
\end{proof}

\begin{remark}\label{rem:end-branches}
{\em We note that Theorem \ref{thm:main5} concerns only the bounded intervals between strictly positive zeros. If $\nu=2k+1$ is an odd integer with $k\geq1$, then $D_\nu(0)=0$ and $(0,d_{\nu,1})$ is another pole-to-pole interval. The proof of the above theorem applies to it without change: it contains a unique strict minimum. If its
value is denoted by $\delta_{\nu,0}$, then for $N\geq2$ we obtain that $\delta_{\nu,0}>\delta_{\nu,1}>\ldots>\delta_{\nu,N-1}>0.$ However, if $\nu>1$ is not an odd integer, $D_\nu(0)\neq0$ and $\Pi_\nu$ extends continuously to the origin, with
\[ \Pi_\nu(0)=1+\frac1\nu\left(\frac{D_\nu'(0)}{D_\nu(0)}\right)^2\geq1.\]
On $[0,d_{\nu,1})$ its minimum is either the endpoint value at $0$ or a unique interior minimum. When $N\geq2$, this initial minimum is strictly larger than $\delta_{\nu,1}$: this is immediate from the displayed endpoint bound if the minimum occurs at $0$, and otherwise follows from the strictly decreasing property of $\zeta_\nu$ and the fact that the initial critical point lies to the left of $\gamma_{\nu,1}$. The unbounded interval $(d_{\nu,N},\infty)$ is not part of the finite sequence in the theorem.}
\end{remark}

\section{\bf An inclusion property of the Tur\'anian of Bessel functions in the complex plane}
\setcounter{equation}{0}

The surprising use of hypergeometric functions in the proof of Bieberbach conjecture by L. de Branges in 1985 has prompted renewed interest in these classes of functions. By using the method of differential subordinations, S.S. Miller and P.T. Mocanu \cite{MillerMocanu1990} determined conditions for the univalence, starlikeness and convexity of Gaussian and confluent hypergeometric functions, and in addition obtained a subordination result for these classes of functions. Motivated by the results in \cite{MillerMocanu1990} on the inclusion (or subordination) property of confluent Kummer and Gaussian hypergeometric functions, S. Andr\'as and \'A. Baricz \cite{AndrasBaricz2009} proved that for $p>q\geq-1/4$ the inclusion $\mathcal{J}_p(\mathbb{D})\subset\mathcal{J}_q(\mathbb{D})$ of the egg-shaped domains is valid, where $$\mathcal{J}_p(z)=2^p\Gamma(p+1)z^{-p}J_p(z),$$ $J_p$ stands for the Bessel function of the first kind, as before, and $\mathbb{D}=\{z\in\mathbb C:|z|<1\}$ is the open unit disk in the complex plane. Moreover, they conjectured that the same property holds true for $p>q>-1.$ This conjecture has been verified recently by L.-I. Cot\^irl\u{a} and R. Sz\'asz \cite{CotirlaSzasz2024} by using an approach based on subordinating factor sequence, which was introduced by H.S. Wilf \cite{Wilf1961} in 1961 for convex conformal mappings on the unit circle. The above natural inclusion (or subordination) property of Bessel functions of the first kind created a considerable interest and recently it has been shown by various researchers that Lommel functions, Mittag-Leffler functions, Wright functions, Le Roy-type hypergeometric functions, Struve functions, $q$-Bessel functions, $q$-Struve functions and a Ramanujan type entire function, for example, satisfy a similar inclusion property. For more details the interested readers can view the papers \cite{AlenaziMehrez2025}, \cite{AlenaziMehrez2026}, \cite{CotirlaSzasz2023}, \cite{CotirlaSzasz2024}, \cite{DenizSzasz2024}, \cite{KumariPrajapat2026}, \cite{MehrezAlenazi2026}, \cite{OzkanKorkmazDeniz2025} and \cite{OzkanKorkmazDenizCaglar2026}, and the references therein.

Motivated by the above mentioned results on various special functions and by the results proved in Section 2 of the present paper, in this section our aim is to show that the Tur\'anian of Bessel functions of the first kind naturally satisfies a similar inclusion (or subordination) result. The technique used in this section is also based on subordinating factor sequence, however, as we can see below the case of the Tur\'anian of Bessel functions of the first kind is more complicated and consequently the proof is more technical. Moreover, we also show what is in the heart of all these inclusion properties of various special functions, for more details see Remark \ref{remliter}.

We first recall some basic definitions. Let $f$ and $g$ two functions, which are analytic in $\mathbb{D}.$ We say that $f$ is {\em subordinate} to $g$ and write $f\prec g$ or $f(z)\prec g(z)$ if there exists a Schwarz function $\omega(z),$ which maps $\D$ into $\D$ and $\omega(0)=0,$ such that $f(z)=g(\omega(z))$ for all $z\in\mathbb{D}.$ If $g$ is univalent (locally injective/one-to-one) on $\D$ and $f(0)=g(0),$ then $f\prec g$ is equivalent to the simpler condition $f(\D)\subset g(\D).$ Now, according to H.S. Wilf
\cite[p. 690]{Wilf1961}, a sequence $\{b_n\}_{n\ge1}$ is called a {\em subordinating factor sequence} if, for every convex univalent function $f(z)=a_1z+a_2z^2+\ldots+a_nz^n+\ldots,$ $a_1\neq0,$ we have $g\prec f,$ where $g(z)=a_1b_1z+a_2b_2z^2+\ldots+a_nb_nz^n+\ldots,$ $z\in\D.$ It is known (see \cite[Theorem 2]{Wilf1961}) that a sequence $\{b_n\}_{n\ge1}$ is a subordinating factor sequence if and only if the real part of $\frac{1}{2}+b_1z+b_2z^2+\ldots+b_nz^n+\ldots$ is strictly positive for all $z\in\D.$

Now, for $\nu\ge0$ consider the Tur\'anian of Bessel functions of the first kind
\[\Delta_\nu(z)=J_\nu^2(z)-J_{\nu-1}(z)J_{\nu+1}(z)\]
and its normalization
$$\mathcal{T}_\nu(z)=2^{2\nu}\Gamma(\nu+1)\Gamma(\nu+2)z^{-2\nu}\Delta_\nu(z),$$
which in view of \cite[eq. (16)]{MezoBaricz2017} can be rewritten as
$$\mathcal{T}_\nu(z)={}_1F_2\left(\nu+\frac12;\nu+2,2\nu+1;-z^2\right).$$
The aim of this section is to prove the inclusion property $\mathcal{T}_p(\D)\subset\mathcal{T}_q(\D)$ for $p>q\geq0.$ As we shall see below, the key step in the proof of the above subordination property is a verification of Wilf's condition (related to subordinating factor sequences) for the hypergeometric multiplier
\begin{equation}\Phi_{p,q}(z)={}_4F_3\!\left(
\begin{matrix}
1,\ p+\frac12,\ q+2,\ 2q+1\\
q+\frac12,\ p+2,\ 2p+1
\end{matrix};z
\right)=\sum_{n\geq0}\beta_n(p,q)z^n,
\label{eq:PhiDef}
\end{equation}
where in terms of the shifted factorial $(\alpha)_n=\alpha(\alpha+1)\ldots(\alpha+n-1)=\Gamma(\alpha+n)/\Gamma(\alpha)$ we have
$$\beta_n(p,q)=\frac{A_n(p)}{A_n(q)}=\frac{(p+\frac12)_n(q+2)_n(2q+1)_n}{(q+\frac12)_n(p+2)_n(2p+1)_n}\quad \mbox{and}\quad
A_n(x)=\frac{(x+\frac12)_n}{(x+2)_n(2x+1)_n n!}.$$

In order to prove the main result of this section we start with the following preliminary result.

\begin{lemma}\label{lem:gpositive}
If for $s\ge0,$ $\alpha\in(0,2\pi)$ and $t\in(0,1)$ we consider
$$f_s(t)=t^{s+1}+2t^{2s}-t^{s-\frac12}\quad \mbox{and}\quad u_\alpha(t)=\frac{1+t}{1-2t\cos\alpha+t^2},$$
then
$$g_s(\alpha)=\int_0^1 f_s(t)u_\alpha(t){\rm d}t>0$$
for all $s\geq 0$ and $\alpha\in(0,2\pi).$
\end{lemma}

\begin{proof}[\bf Proof of Lemma \ref{lem:gpositive}]
If $a=s+1/2\ge1/2,$ then $f_s(t)=t^{a-1}\left(t^{3/2}+2t^a-1\right).$ The function $t\mapsto v_a(t)=t^{3/2}+2t^a-1$
is strictly increasing on $(0,1)$ with $v_a(0^+)=-1$ and $v_a(1)=2.$ Consequently, there is a unique $\tau\in(0,1)$ such that $f_s(t)<0$ for all $0<t<\tau$ and $f_s(t)>0$ for all $\tau<t<1.$ Now, we consider the following integrals
\[\varepsilon_1=-\int_0^\tau f_s(t){\rm d}t\quad \mbox{and} \quad \varepsilon_2=\int_\tau^1 f_s(t){\rm d}t.\]
Observe that
$$\varepsilon_2-\varepsilon_1=\int_0^1 f_s(t){\rm d}t=\frac1{s+2}=\frac1{a+\frac32}.$$
On the other hand, for $0<t<\tau$ we have $-f_s(t)<t^{a-1}(1-2t^a),$ which implies that
$$\varepsilon_1<\int_0^{\tau}t^{a-1}(1-2t^a){\rm d}t=\frac{1}{a}\tau^a(1-\tau^a)<\frac{1}{4a}\le\frac1{a+\frac32},$$
where we used that $a\ge1/2$ and the fact that since $\tau^{3/2}+2\tau^a=1$ we have $0<\tau^a<1/2.$ Thus, we arrive at $\varepsilon_1<\varepsilon_2-\varepsilon_1,$ which in turn implies that
$$g_s(\pi)=\int_0^\tau\frac{f_s(t)}{1+t}{\rm d}t+\int_\tau^1\frac{f_s(t)}{1+t}{\rm d}t>-\varepsilon_1+\frac12\varepsilon_2>0.$$
Now, we fix $0<\alpha<2\pi$ and we define
$$r_\alpha(t)=\frac{u_\alpha(t)}{u_\pi(t)}=\frac{(1+t)^2}{1-2t\cos\alpha+t^2}.$$
Since for all $\alpha\in(0,2\pi)$ and $t\in(0,1)$ we have
$$r_\alpha'(t)=\frac{2(1+\cos\alpha)(1-t)(1+t)}{(1-2t\cos\alpha+t^2)^2}\ge0,$$
clearly $r$ is nondecreasing, and consequently for all $\alpha\in(0,2\pi)$ and $t\in(0,1)$ we obtain that
$$f_s(t)\left(r_\alpha(t)-r_\alpha(\tau)\right)\ge0$$
and
$$g_s(\alpha)=r_\alpha(\tau)g_s(\pi)+\int_0^1f_s(t)u_\pi(t)\left(r_\alpha(t)-r_\alpha(\tau)\right){\rm d}t>0.$$
Here we used tacitly that, since $1-2\tau\cos\alpha+\tau^2=\left|1-\tau e^{\mathrm{i}\alpha}\right|^2>0,$ we have $r_{\alpha}(\tau)>0.$
\end{proof}

Now, let us recall the definition of positive and conditionally negative
definite functions on $\mathbb{Z}.$ By definition, an even function $\varphi:\mathbb{Z}\to\mathbb{R}$ is {\em positive definite} if
\[\sum_{j,k=1}^N c_j\overline{c_k}\varphi(m_j-m_k)\ge0\]
for every $N$ positive integer, every $m_1,\ldots,m_N\in\mathbb{Z}$ and every $c_1,\ldots,c_N\in\mathbb{C}.$ Moreover, an even function
$\psi:\mathbb{Z}\to\mathbb{R}$ with $\psi(0)=0$ is {\em conditionally negative definite} if
\[\sum_{j,k=1}^N c_j\overline{c_k}\psi(m_j-m_k)\le0\]
whenever $\sum_{j=1}^N c_j=0.$

In view of the preceding lemma, we arrive to the following important preliminary result.

\begin{lemma}\label{prop:CND}
If for every $x\ge0$ and $n\ge0$ we consider the expression
\[\lambda_n(x)=-\frac{\partial}{\partial x}\log A_n(x),\quad \lambda_0(x)=0,\]
then the even extension $n\mapsto\lambda_{|n|}(x)$ for $n\in\mathbb{Z}$ is conditionally negative definite. Consequently, for $p>q\geq0$ the even sequence
$n\mapsto\beta_{|n|}(p,q)$ for $n\in\Z$ is positive definite.
\end{lemma}

\begin{proof}[\bf Proof of Lemma \ref{prop:CND}]
Observe that differentiating $\log A_n(x)$ with respect to $x$ we arrive at
$$\lambda_n(x)=\sum_{k=0}^{n-1}\left(\frac1{k+x+2}+\frac2{k+2x+1}-\frac1{k+x+\frac12}\right)=\int_0^1
\frac{1-t^n}{1-t}f_x(t){\rm d}t.$$
On the other hand, recall that the well-known Poisson kernel
\begin{equation}\label{kernPoisson}P_t(\alpha)=\frac{1-t^2}{1-2t\cos\alpha+t^2}=1+2\sum_{m\geq1}t^m\cos(m\alpha)\end{equation}
satisfies the next two standard Fourier coefficient identities
$$\frac{1}{2\pi}\int_0^{2\pi}P_t(\alpha){\rm d}\alpha=1\quad \mbox{and}\quad \frac{1}{2\pi}\int_0^{2\pi}P_t(\alpha)\cos(n\alpha){\rm d}\alpha=t^n,$$
where the second identity follows from the Fourier series expansion and the orthogonality of Chebyshev polynomials of the first kind. These identities imply that $$\frac1{2\pi}\int_0^{2\pi}(1-\cos(n\alpha))u_\alpha(t){\rm d}\alpha=\frac{1-t^n}{1-t}.$$
Combining the above relations we arrive at
$$\lambda_n(x)=\frac1{2\pi}\int_0^{2\pi}(1-\cos(n\alpha)) g_x(\alpha){\rm d}\alpha,$$
where the interchange of the integrals is justified
after multiplication by $1-\cos(n\alpha)$, since
\[\int_0^1 \left|f_x(t)\right|\frac{1-t^n}{1-t}{\rm d}t<\infty.\]
Now, let $m_1,\ldots,m_N\in\Z$ and $c_1,\ldots,c_N\in\mathbb{C}$ satisfy
$\sum_{j=1}^Nc_j=0.$ In view of the above representation of $\lambda_n(x)$ we arrive at
\begin{align*}
\sum_{j,k=1}^Nc_j\overline{c_k}\lambda_{|m_j-m_k|}(x)&=\frac{1}{2\pi}
\int_0^{2\pi}g_x(\alpha)\sum_{j,k=1}^{N}c_j\overline{c_k}\left(1-\cos\left((m_j-m_k)\alpha\right)\right){\rm d}\alpha\\
&=-\frac1{2\pi}\int_0^{2\pi}g_x(\alpha)\left|\sum_{j=1}^N c_je^{\mathrm{i}m_j\alpha}\right|^2{\rm d}\alpha\le0,
\end{align*}
because $g_x(\alpha)>0$ by Lemma \ref{lem:gpositive}. Thus $n\mapsto\lambda_{|n|}(x)$ is conditionally negative definite.

Now, observe that by definition of $\beta_n(p,q)$ we have
$$-\log\beta_n(p,q)=\int_q^p\lambda_n(x){\rm d}x.$$
Since the cone of conditionally negative definite functions is closed under positive integration, $n\mapsto\Lambda(n):=-\log\beta_{|n|}(p,q)$ is conditionally negative definite on $\Z$ with $\Lambda(0)=0.$ On the other hand, Schoenberg's theorem (see for example \cite[Proposition 4.4]{SchillingSongVondracek2012}) states that if the function $\psi$ is conditionally negative definite, then $\psi(0)\geq0$ and $s\mapsto e^{-t\psi(s)}$ is positive definite for every $t>0.$ Thus, Schoenberg's theorem now implies that $n\mapsto e^{-t\Lambda(n)}$ is positive definite on $\mathbb{Z}$ for all $t>0,$ and in particular $n\mapsto e^{-\Lambda(n)}=\beta_{|n|}(p,q)$ is positive definite on $\Z$ for all $p>q\geq0.$
\end{proof}

We can now verify Wilf's condition for the hypergeometric multiplier.

\begin{lemma}\label{lemWilfcond}
If $p>q\geq0,$ then the real part of $\frac{1}{2}+\beta_1(p,q)z+\beta_2(p,q)z^2+\ldots+\beta_n(p,q)z^n+\ldots$ is strictly positive for all $z\in\D,$ and
consequently the sequence $\{\beta_n(p,q)\}_{n\ge1}$ is a subordinating factor sequence.
\end{lemma}

\begin{proof}[\bf Proof of Lemma \ref{lemWilfcond}]
By the classical Herglotz theorem (see for example \cite[p. 39]{Katznelson2004}) a positive definite sequence $\{\varphi(n)\}_{n\in\Z}$ with $\varphi(0)=1$ is the Fourier sequence of a probability measure on the unit circle. In view of Lemma \ref{prop:CND} and the Herglotz theorem there exists a probability measure $\mu_{p,q}$ on $[0,2\pi)$ such that
$$\beta_n(p,q)=\int_0^{2\pi}e^{-\mathrm{i}n\alpha}{\rm d}\mu_{p,q}(\alpha),\quad n\in\Z,$$
and consequently
$$\Phi_{p,q}(z)=\sum_{n\geq0}\left(\int_0^{2\pi}e^{-\mathrm{i}n\alpha}{\rm d}\mu_{p,q}(\alpha)\right)z^n=
\int_0^{2\pi}\sum_{n\geq0}\left(ze^{-\mathrm{i}\alpha}\right)^n{\rm d}\mu_{p,q}(\alpha)=
\int_0^{2\pi}\frac{{\rm d}\mu_{p,q}(\alpha)}{1-ze^{-\mathrm{i}\alpha}}.$$
Now, since $\mu_{p,q}$ is a probability measure, we arrive at
$$2\Phi_{p,q}(z)-1=1+2\sum_{n\geq1}\beta_n(p,q)z^n=\int_0^{2\pi}\frac{1+ze^{-\mathrm{i}\alpha}}{1-ze^{-\mathrm{i}\alpha}}{\rm d}\mu_{p,q}(\alpha),$$
which in turn implies that for all $z\in\D$ we have
$$\real\left(2\Phi_{p,q}(z)-1\right)=\int_0^{2\pi}\frac{1-|z|^2}{|1-ze^{-\mathrm{i}\alpha}|^2}{\rm d}\mu_{p,q}(\theta)>0$$
and thus, according to \cite[Theorem 2]{Wilf1961}, the sequence $\{\beta_n(p,q)\}_{n\ge1}$ is indeed a subordinating factor sequence.
\end{proof}

Now, we define the function
$$F_\nu(z)={}_1F_2\left(\nu+\frac12;\nu+2,2\nu+1;-z\right)=1+\sum_{n\geq1}(-1)^nA_n(\nu)z^n.$$
Observe that $\mathcal{T}_\nu(z)=F_\nu(z^2)$ and since $z\mapsto z^2$ maps $\D$ onto $\D,$ we have $\mathcal{T}_\nu(\D)=F_\nu(\D).$

We now focus on the convexity needed in the subordination argument of H.S. Wilf.

\begin{lemma}\label{lem:convex}
For every $q\ge0,$ the function
\[G_q(z)=\frac{1-F_q(z)}{A_1(q)}\]
is convex and univalent in $\D.$
\end{lemma}

\begin{proof}[\bf Proof of Lemma \ref{lem:convex}]
Since
\[A_1(q)=\frac1{2(q+2)},\]
we clearly have
\[G_q(z)=z+\sum_{n\geq2}(-1)^{n+1}c_n(q)z^n,\quad \mbox{where}\quad c_n(q)=\frac{A_n(q)}{A_1(q)}.\]
Now, for $n\ge1$ we have
\[\frac{c_{n+1}(q)}{c_n(q)}=\frac{n+q+\frac12}{(n+q+2)(n+2q+1)(n+1)}\]
and we claim that this quotient is at most $1/8.$ Indeed, with $n=m+1$ we arrive at
\begin{align*}
&(n+q+2)(n+2q+1)(n+1)-8\left(n+q+\frac12\right)\\
&=m^3+(3q+7)m^2+(2q^2+14q+8)m+4q^2+8q\ge0.
\end{align*}
and consequently $c_n(q)\le8^{1-n}$ for all $n\ge1$ and $q\geq0.$ On the other hand, in view of the power series representation of $G_q(z)$ and the triangle inequality, observe that for $|z|<1$ we have
$$\left|G_q'(z)-1\right|+\left|zG_q''(z)\right|\leq\sum_{n\geq2}n^2c_n(q)|z|^{n-1}<\sum_{n\geq2}\frac{n^2}{8^{n-1}}=\frac{233}{343}<1,$$
which in turn implies that for all $z\in\D$ and $q\geq0$ we have
\[|zG_q''(z)|<1-|G_q'(z)-1|\le |G_q'(z)|,\]
where the second inequality follows from the reversed triangle inequality. In other words, we have
\[\left|\frac{zG_q''(z)}{G_q'(z)}\right|<1,\]
and therefore
\[\real\left(1+\frac{zG_q''(z)}{G_q'(z)}\right)>0,\]
which is the standard analytic characterization for convexity.
\end{proof}

We arrive at the main theorem of this section. Figure \ref{Fig3} illustrates this result.

\begin{theorem}\label{thm:subordination}
If $p>q\geq0,$ then $F_p\prec F_q$ in $\D,$ and consequently $\mathcal{T}_p(\D)\subseteq\mathcal{T}_q(\D).$
\end{theorem}

\begin{figure}[t]
	\centering
	\includegraphics[width=0.9\textwidth]{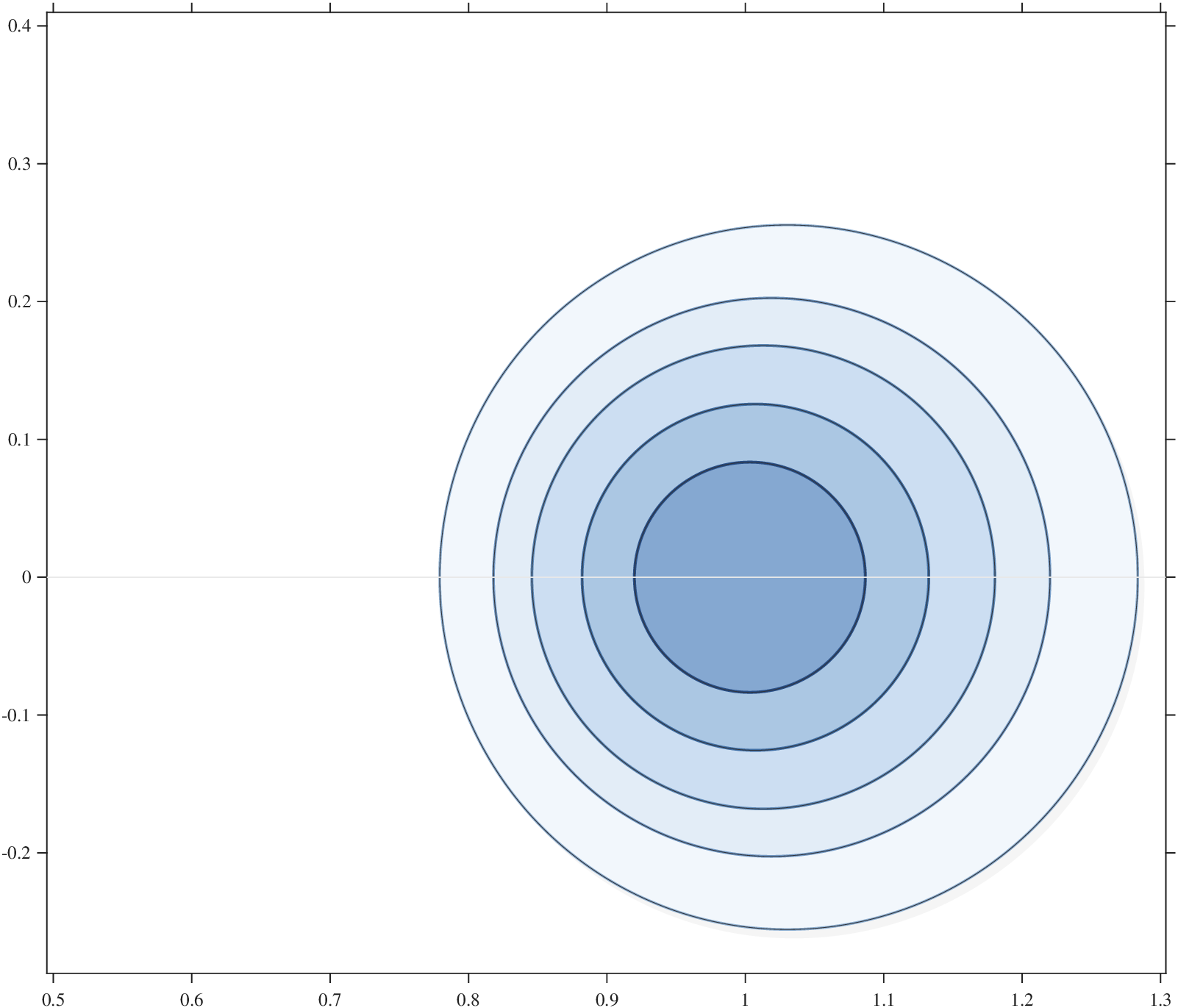}
	\caption{Numerical illustration of the image domains of the normalized Bessel Tur\'anian for $\nu\in\left\{0,\frac12,1,2,4\right\}$. The boundaries are parametrized by $w=F_\nu(e^{it})$, $0\le t\le 2\pi$, and illustrate $\mathcal{T}_4(\mathbb D)\subset \mathcal{T}_2(\mathbb D)\subset \mathcal{T}_1(\mathbb D)\subset \mathcal{T}_{1/2}(\mathbb D)\subset \mathcal{T}_0(\mathbb D).$}
	\label{Fig3}
\end{figure}

\begin{proof}[\bf Proof of Theorem \ref{thm:subordination}]
Since
\[A_n(p)=\beta_n(p,q)A_n(q),\]
we clearly have
$$\frac{1-F_p(z)}{A_1(q)}=G_q(z)*\left(\sum_{n\geq1}\beta_n(p,q)z^n\right),$$
where $*$ denotes the Hadamard product of two power series. By Lemma \ref{lem:convex} the function $G_q$ is convex and univalent, and by
Lemma \ref{lemWilfcond} the sequence $\{\beta_n(p,q)\}_{n\ge1}$ is a subordinating factor sequence. Hence the definition of the subordinating factor sequences and the above Hadamard product representation implies the next subordination
\[\frac{1-F_p}{A_1(q)}\prec\frac{1-F_q}{A_1(q)}.\]
Since multiplication by $A_1(q)$ and subtraction from $1$ preserve the same Schwarz function, we arrive at
$F_p\prec F_q$ and consequently $\mathcal{T}_p(\D)=F_p(\D)\subseteq F_q(\D)=\mathcal{T}_q(\D).$
\end{proof}

\begin{remark}
{\em We note that Theorem \ref{thm:subordination} is stronger than the sufficient criterion obtained by requiring $\{\beta_n(p,q)\}_{n\geq0}$ to be a convex null sequence. The latter imposes an additional algebraic restriction on $p$ and $q,$ whereas the positive definiteness argument applied above proves Wilf's
condition for every $p>q\geq0.$ More precisely, it is possible to show that for $p>q\geq0$ the sequence $\{\beta_n(p,q)\}_{n\geq0}$ is strictly decreasing. Indeed, if we denote $r_n(p,q)={\beta_{n+1}(p,q)}/{\beta_n(p,q)},$ then
$$r_n(p,q)=\frac{(n+p+\frac12)(n+q+2)(n+2q+1)}{(n+q+\frac12)(n+p+2)(n+2p+1)}$$
and
$$1-r_n(p,q)=\frac{(p-q)\left(4n^2+4n(p+q)+7n+4pq+2p+2q+1\right)}{(n+p+2)(n+2p+1)(2n+2q+1)}>0$$
for all $p>q\geq0$ and $n\geq0.$ Thus $0<r_n(p,q)<1$ and hence $\beta_{n+1}(p,q)<\beta_n(p,q)$ for all $p>q\geq0$ and $n\geq0.$ Moreover, it is also possible to show that the elements of this sequence satisfy the convexity property $$\beta_n(p,q)-2\beta_{n+1}(p,q)+\beta_{n+2}(p,q)>0$$ for all $n\geq1$ and $p>q\geq0.$ Indeed, if use the notation $p=q+d$ with $d>0$ and choose $n=m+1$, $m\ge0,$ then after clearing the positive denominator in $r_{n+1}(p,q)-r_n(p,q)$, the numerator will be $d\cdot \varphi_{q,d}(m)$, where $\varphi_{q,d}(m)$ is
\begin{align*}
8&m^4+(16d+32q+76)m^3+(8d^2+48dq+102d+48q^2+204q+257)m^2\\
&+\left(16d^2q+32d^2+48dq^2+168dq+194d+32q^3+168q^2+388q+357\right)m\\
&+8d^2q^2+20d^2q+24d^2+16dq^3+60dq^2+98dq+96d+8q^4+40q^3+98q^2+192q+162.
\end{align*}
Since every coefficient is positive, we clearly have $r_{n+1}(p,q)>r_n(p,q)$ for all $n\ge1$ and $p>q\geq0.$ This in turn implies that
$$\beta_n(p,q)-2\beta_{n+1}(p,q)+\beta_{n+2}(p,q)=\beta_n(p,q)\left(\left(1-r_n(p,q)\right)^2+r_n(p,q)(r_{n+1}(p,q)-r_n(p,q))\right)>0$$
for all $n\geq 1$ and $p>q\geq0.$ Now, suppose that we would like to apply L. Fej\'er's classical result \cite[p. 80]{Fejer1925}, which says that for the nonnegative sequence $\{c_n\}_{n\geq0}$ the sufficient condition for the real part of $c_0/2+c_1z+c_2z^2+\ldots+c_nz^n+\ldots$ to be strictly positive is that the sequence must satisfy $c_n-c_{n+1}\geq c_{n+1}-c_{n+2}\geq0$ for all $n\geq0.$ Thus we would need the condition $\beta_0(p,q)-2\beta_{1}(p,q)+\beta_{2}(p,q)>0,$ which is equivalent to
$$(2q+3)(p-q)^2+(2q^2+7q+6)(p-q)+q^2-3\ge0.$$
Thus, by using L. Fej\'er's result together with Lemma \ref{lem:convex}, we are able to prove Theorem \ref{thm:subordination} only for the case $p>q\geq\sqrt{3}.$}
\end{remark}

\begin{remark}\label{remliter}
{\em Recall that by definition a sequence $\{a_n\}_{n\geq 0}$ of nonnegative real numbers is called a {\em Hausdorff moment sequence} if there exists a finite positive Borel
measure $\rho$ on $[0,1]$ such that for all $n\geq 0$
\[a_n=\int_0^1 t^n{\rm d}\rho(t).\]
If $a_0=1$, then $\rho$ is a probability measure. Moreover, it is known that if the sequence $\{a_n\}_{n\geq 0}$ is a Hausdorff moment sequence, then the even extension $m\mapsto a_{|m|}$ is positive definite on $\mathbb Z$. To see this observe that for $m\in\mathbb Z$,
\[\widetilde a(m)=a_{|m|}=\int_0^1 t^{|m|}{\rm d}\rho(t)\]
and for each fixed $t\in[0,1]$, the sequence $m\mapsto t^{|m|}$ is positive definite on $\mathbb Z$, since
\[t^{|m|}=\frac{1}{2\pi}\int_0^{2\pi}P_t(\alpha)e^{-{\rm i}m\alpha}{\rm d}\alpha,\]
where $P_t(\alpha)$ is the Poisson kernel, defined by \eqref{kernPoisson}. Hence, for arbitrary $m_1,\ldots,m_N\in\mathbb Z$ and
$c_1,\ldots,c_N\in\mathbb C$ we have
\[\sum_{j,k=1}^Nc_j\overline{c_k}\widetilde a(m_j-m_k)=\int_0^1\sum_{j,k=1}^Nc_j\overline{c_k}t^{|m_j-m_k|}{\rm d}\rho(t)\ge0\]
and therefore $\widetilde a$ is positive definite on $\mathbb Z$.

Now, it is important to note here that a similar proof as we presented above for the Tur\'anian of Bessel functions of the first kind can be used for Bessel functions of the first kind itself. More precisely, if we focus on subordinating factor sequence part presented in Lemma \ref{lemWilfcond}, then we have the following observations concerning some of the published results in the literature. We note here that in the next cases the conditionally negative definiteness of the sequences in question is not required, only the positive definiteness.
\begin{enumerate}
\item[\bf A.] In the case of Bessel functions of the first kind the multiplier sequence (see \cite[Lemma 2]{CotirlaSzasz2024}) is $b_n(p,q)=(q+1)_n/(p+1)_n,$ which is a Hausdorff moment sequence (and in particular, a beta moment sequence), and the proof concerning the fact that $\left\{b_n(p,q)\right\}$ is a subordinating factor sequence is based on the following observation: the sequence $m\mapsto t^{|m|}$ is positive definite on $\Z$ for all $t\in[0,1],$ and consequently the sequence $m\mapsto b_{|m|}(p,q)$ is positive definite on $\mathbb{Z}.$
\item[\bf B.] In the case of Lommel functions of the first kind the multiplier sequence (see \cite[Lemma 2.1]{CotirlaSzasz2023}) has the general term $b_n=(c)_n(d)_n/((a)_n(b)_n)$ for $a>c>0$ and $b>d>0,$ and is also a Hausdorff moment sequence since it is the product of the Hausdorff moment sequences $\left\{(c)_n/(a)_n\right\}$ and $\left\{(d)_n/(b)_n\right\},$ and similarly the sequence $m\mapsto b_{|m|}$ is positive definite on $\mathbb{Z}.$
\item[\bf C.] In the case of Struve functions one of the multiplier sequences (see \cite[Lemma 2.5]{DenizSzasz2024}) has the general term $f_n=\Gamma(n+\nu+3/2)/\Gamma(n+\mu+3/2)$ for $\mu>\nu>0$ and the other one (see \cite[Lemma 2.6]{DenizSzasz2024}) has the general term $g_n=\Gamma(\nu+1+n/2)/\Gamma(\mu+1+n/2)$ for $\mu>\nu>0.$ Since the sequences $1,f_1,f_2,\ldots$ and $1,g_1,g_2,\ldots$ are Hausdorff moment sequences, it follows that the even functions $m\mapsto f_{|m|}$ and $m\mapsto g_{|m|}$ are positive definite on $\mathbb{Z}.$ However, it should be mentioned here that in \cite{KumariPrajapat2026} the result in \cite{DenizSzasz2024} was corrected and instead of the multiplier sequence $\left\{\Gamma(n+\nu+3/2)/\Gamma(n+\mu+3/2)\right\}$ its shifted form $\left\{\Gamma(n+\nu+1/2)/\Gamma(n+\mu+1/2)\right\}$ should be used.
\item[\bf D.] In the case of Wright functions the general term of the multiplier sequence is defined by $v_n=\Gamma(\alpha n+\beta)/\Gamma(\alpha n+\mu)$ for $n\geq1$ and $v_0=1,$ where $\alpha>0,$ $\mu>\beta>x^*,$ $x^*\simeq1.461632$ is the abscissa of the minimum of the gamma function, see \cite[Proposition 3.1]{AlenaziMehrez2025} for more details. By using the beta integral it can be shown that $\{1,v_1,v_2,\ldots\}$ is a Hausdorff moment sequences, which implies that the even function $m\mapsto v_{|m|}$ is positive definite on $\mathbb{Z}.$
\item[\bf E.] In the case of Le Roy-type hypergeometric functions the multiplier sequence has the general term $(\Gamma(an+A))^{\alpha}/(\Gamma(bn+B))^{\beta}$ for certain conditions on the parameters, see \cite[Proposition 3.2]{AlenaziMehrez2026}. In the special case when $a=b,$ $\alpha=\beta$ and $B>A>0$ our argument concerning the positive definiteness is also working here.
\item[\bf F.] In the case of Jackson $q$-Bessel functions the multiplier sequence is defined by $f_0=1,$ $f_n=q^{n(\mu-\nu)}\Gamma_q(\mu+1)\Gamma_q(\nu+n+1)/(\Gamma_q(\nu+1)\Gamma_q(\mu+n+1)),$ where $n\geq1,$ $\mu>\nu>-1,$ see \cite[Lemma 3.3]{OzkanKorkmazDeniz2025} for more details. It can be shown that $\{1,f_1,f_2,\ldots\}$ is a Hausdorff moment sequences, which implies that the even function $m\mapsto f_{|m|}$ is positive definite on $\mathbb{Z}.$
\end{enumerate}}
\end{remark}

\section{\bf Discussion and future work}

In this paper, by using the corresponding Riccati differential equations of the logarithmic derivative of special functions like Bessel functions of the first kind, general cylinder functions, regular Coulomb wave functions, modified Bessel functions of the second kind of purely imaginary order and parabolic cylinder functions, we obtained sharp Tur\'an type inequalities for these oscillatory special functions. An essential ingredient of the proofs was the fact that the corresponding normalized Tur\'anians are related to the logarithmic derivatives. Moreover, we have shown an inclusion property of another normalized Tur\'anian of Bessel functions of the first kind in the complex plane.

It is expected that the techniques employed in the paper (in Sections 2--6) could be useful to treat similar problems related to other special functions, such as the Struve functions, Lommel functions, Tricomi confluent hypergeometric functions or the $q$-Bessel functions. Apparently, it can be happen that the methods of the present paper are not directly applicable in these cases, however, we believe that by an appropriate reformulation of the problems or modification of our techniques, they can be tackled successfully. Moreover, we also believe that the monotony properties of other special functions in some cases may be studied more efficiently by using the approach of Section 7.

In addition, we list below a few interesting open questions that are worth addressing and that arose during the writing of this study:

\begin{enumerate}
\item[\bf Q1.] Concerning the case of Bessel functions of the first kind in Section 2, is it possible to determine the sign of the second derivative of $\beta_{\mu,n}$ with respect to $\mu$? For each $n$, does $\beta_{\mu,n}''$ changes sign exactly once?
\item[\bf Q2.] In the case when the nodal interval left endpoint lies to the left of the turning point $x=\nu$ is the function $\mu\mapsto b_{\mu,n}$ decreasing on $(0,\infty)$ for each $\alpha\in(0,\pi)$ and $n\in\mathbb{N}$ fixed? The monotonicity fails for some phases? Here $b_{\mu,n}$ is the corresponding minimum value of the normalized Tur\'anian of general Bessel functions in Section 3.
\item[\bf Q3.] Concerning the case of regular Coulomb wave functions in Section 4, is the function $L\mapsto b_{L,\eta,n}$ decreasing on $(0,\infty)$ for each $\eta>0$ and $n\in\mathbb{N}$ fixed? For fixed $(\eta,n),$ does $b_{L,\eta,n}$ have exactly one turning point as a function of $L$?
\item[\bf Q4.] For a continuously varying gap between two adjacent real zero branches of $D_\nu,$ how does the optimal Tur\'an constant $\delta_{\nu,n}$ depend on $\nu$? Is it always strictly decreasing along the branch? Here $\delta_{\mu,n}$ is the corresponding minimum value of the normalized Tur\'anian of parabolic cylinder functions in Section 6.
\end{enumerate}


\begin{thebibliography}{width}

\bibitem[\bf AM2025]{AlenaziMehrez2025}
\textsc{A. Alenazi, K. Mehrez}, A note on the monotonicity properties of two classes of analytic functions related to the Mittag-Leffler and Wright functions, {\em Results Math.} 80(7) (2025) Art. 216.

\bibitem[\bf AM2026]{AlenaziMehrez2026}
\textsc{A. Alenazi, K. Mehrez}, Convexity and monotonicity of analytic functions related to Le Roy-type hypergeometric functions, {\em Ukrainian Math. J.} 78(3) (2026) Art. 59.

\bibitem[\bf AB2009]{AndrasBaricz2009}
\textsc{S. Andr\'as, \'A. Baricz}, Monotonicity property of generalized and normalized Bessel functions of complex order, {Complex Var. Elliptic Equ.} 54(7) (2009) 689--696.

\bibitem[\bf Ba2010]{Baricz2010}
\textsc{\'A. Baricz}, Tur\'an type inequalities for modified Bessel functions, {\em Bull. Aust. Math. Soc.} 82(2) (2010) 254--264.

\bibitem[\bf Ba2015a]{Baricz2015}
\textsc{\'A. Baricz}, Bounds for Tur\'anians of modified Bessel functions, {\em Expo. Math.} 33(2) (2015) 223--251.

\bibitem[\bf Ba2015b]{Baricz2015C}
\textsc{\'A. Baricz}, Tur\'an type inequalities for regular Coulomb wave functions, {\em J. Math. Anal. Appl.} 430(1) (2015) 166--180.

\bibitem[\bf BI2013]{BariczIsmail2013}
\textsc{\'A. Baricz, M.E.H. Ismail}, Tur\'an type inequalities for Tricomi confluent hypergeometric functions, {\em Constr. Approx.} 37(2) (2013) 195--221.

\bibitem[\bf BK2016]{BariczKoumandos2016}
\textsc{\'A. Baricz, S. Koumandos}, Tur\'an-type inequalities for some Lommel functions of the first kind, {\em Proc. Edinb. Math. Soc.} 59(3) (2016) 569--579.

\bibitem[\bf BP2014]{BariczPogany2014}
\textsc{\'A. Baricz, T.K. Pog\'any}, Tur\'an determinants of Bessel functions, {\em Forum Math.} {26}(1) (2014) 295--322.

\bibitem[\bf BP2013]{BariczPonnusamy2013}
\textsc{\'A. Baricz, S. Ponnusamy}, On Tur\'an type inequalities for modified Bessel functions, {\em Proc. Amer. Math. Soc.} 141(2) (2013) 523--532.

\bibitem[\bf BPS2016]{BariczPonnusamySingh2016}
\textsc{\'A. Baricz, S. Ponnusamy, S. Singh}, Tur\'an type inequalities for general Bessel functions, {\em Math. Inequal. Appl.} 19(2) (2016) 709--719.

\bibitem[\bf BPS2017]{BariczPonnusamySingh2017}
\textsc{\'A. Baricz, S. Ponnusamy, S. Singh}, Tur\'an type inequalities for Struve functions, {\em J. Math. Anal. Appl.} 445(1) (2017) 971--984.

\bibitem[\bf BPV2011]{BariczPonnusamyVuorinen2011}
\textsc{\'A. Baricz, S. Ponnusamy, M. Vuorinen}, Functional inequalities for modified Bessel functions, {\em Expo. Math.} {29}(4) (2011) 399--414.

\bibitem[\bf BBF2018]{BechererBilarevFrentrup2018}
\textsc{D. Becherer, T. Bilarev, P. Frentrup}, Optimal liquidation under stochastic liquidity, {\em Finance Stoch.} 22(1) (2018) 39--68.

\bibitem[\bf Ch2026]{Chung2026}
\textsc{S.-Y. Chung}, Uniform bounds, zero separation and monotonicity for the regular Coulomb wave functions, {\em J. Math. Anal. Appl.} 557(1) (2026) Art. 130275.

\bibitem[\bf CS2023]{CotirlaSzasz2023}
\textsc{L.-I. Cot\^irl\u{a}, R. Sz\'asz}, The monotony of the Lommel functions, Result. Math. 78 (2023) Art. 127.

\bibitem[\bf CS2024]{CotirlaSzasz2024}
\textsc{L.-I. Cot\^irl\u{a}, R. Sz\'asz}, On the monotony of Bessel functions of the first kind, {\em Comput. Methods Funct. Theory} 24(4) (2024) 747--752.

\bibitem[\bf DS2024]{DenizSzasz2024}
\textsc{E. Deniz, R. Sz\'asz}, On the monotony of Struve functions, {\em Complex Anal. Oper. Theory} 18 (2024) Art. 120.

\bibitem[\bf Fe1925]{Fejer1925}
\textsc{L. Fej\'er}, \"Uber die Positivit\"at von Summen, die nach trigonometrischen oder Legendreschen Funktionen fortschreiten, {\em Acta Sci. Math. (Szeged)} 2 (1925) 75--86.

\bibitem[\bf Ga1971]{Gasper1971}
\textsc{G. Gasper}, On the extension of Tur\'an's inequality to Jacobi polynomials, {\em Duke Math. J.} 38 (1971) 415--428.

\bibitem[\bf Ga1972]{Gasper1972}
\textsc{G. Gasper}, An inequality of Tur\'an type for Jacobi polynomials, {\em Proc. Amer. Math. Soc.} 32 (1972) 435--439.

\bibitem[\bf KK2013]{KalmykovKarp2013}
\textsc{S.I. Kalmykov, D.B. Karp}, Log-convexity and log-concavity for series in gamma ratios and applications, {\em J. Math. Anal. Appl.} 406(2) (2013) 400--418.

\bibitem[\bf KS1960]{KarlinSzego1960}
\textsc{S. Karlin, G. Szeg\H{o}}, On certain determinants whose elements are orthogonal polynomials, {\em J. Analyse Math.} 8 (1960/61) 1--157.

\bibitem[\bf KS2010]{KarpSitnik2010}
\textsc{D. Karp, S.M. Sitnik}, Log-convexity and log-concavity of hypergeometric-like functions, {\em J. Math. Anal. Appl.} 364(2) (2010) 384--394.

\bibitem[\bf Ka2004]{Katznelson2004}
\textsc{Y. Katznelson}, {\em An Introduction to Harmonic Analysis}, Cambridge University Press, Cambridge, 2004.

\bibitem[\bf Ko2020]{Koch2020}
\textsc{T. Koch}, Universal bounds and monotonicity properties of ratios of Hermite and parabolic cylinder functions, {\em Proc. Amer. Math. Soc.} 148(5) (2020) 2149--2155.

\bibitem[\bf KP2026]{KumariPrajapat2026}
\textsc{N. Kumari, J.K. Prajapat}, On the monotony of Struve functions and Lommel functions, {\em J. Anal.} 34(2) (2026) 1215--1224.

\bibitem[\bf La1986]{Laforgia1986}
\textsc{A. Laforgia}, Inequalities and monotonicity results for zeros of modified Bessel functions of purely imaginary order, {\em Quart. Appl. Math.} 44(1) (1986) 91--96.

\bibitem[\bf Lo1990]{Lorch1990}
\textsc{L. Lorch}, Some monotonicity properties associated with the zeros of Bessel functions, {\em Arch. Math. (Brno)} 26(2-3) (1990) 137--146.

\bibitem[\bf MA2026]{MehrezAlenazi2026}
\textsc{K. Mehrez, A. Alenazi}, Convexity of Ramanujan type entire functions and its applications, {\em J. Inequal. Appl.} 2026 (2026) Art. 59.

\bibitem[\bf MB2017]{MezoBaricz2017}
\textsc{I. Mez\H o, \'A. Baricz}, Properties of the Tur\'anian of modified Bessel functions, {\em Math. Inequal. Appl.} 20(4) (2017) 991--1001.

\bibitem[\bf MM1990]{MillerMocanu1990}
\textsc{S.S. Miller, P.T. Mocanu}, Univalence of Gaussian and confluent hypergeometric functions, {\em Proc. Amer. Math. Soc.} 110(2) (1990) 333--342.

\bibitem[\bf MS1984]{MuldoonSpigler1984}
\textsc{M.E. Muldoon, R. Spigler}, Some remarks on zeros of cylinder functions, {\em SIAM J. Math. Anal.} 15(6) (1984) 1231--1233.

\bibitem[\bf NIST2010]{nist}
{\em NIST Handbook of Mathematical Functions}, edited by \textsc{F.W.J. Olver, D.W. Lozier, R.F. Boisvert, C.W. Clark}, National Institute of Standards and Technology, Washington, Cambridge University Press, Cambridge, 2010.

\bibitem[\bf \"OKD2025]{OzkanKorkmazDeniz2025}
\textsc{Y. \"Ozkan, S. Korkmaz, E. Deniz}, The monotony of the $q$-Bessel functions, {\em J. Math. Anal. Appl.} 549(1) (2025) Art. 129439.

\bibitem[\bf \"OKD\c{C}2026]{OzkanKorkmazDenizCaglar2026}
\textsc{Y. \"Ozkan, S. Korkmaz, E. Deniz, M. \c{C}a\u{g}lar}, The monotony of the $q$-Struve-Bessel functions, {\em J. Inequal. Appl.} 2026 (2026) Art. 6.

\bibitem[\bf SSV2012]{SchillingSongVondracek2012}
\textsc{R.L. Schilling, R. Song, Z. Vondra\v{c}ek}, {\em Bernstein Functions: Theory and Applications}, De Gruyter Studies in Mathematics, vol. 37, Berlin, 2012.

\bibitem[\bf Se2011]{Segura2011}
\textsc{J. Segura}, Bounds for ratios of modified Bessel functions and associated Tur\'an-type inequalities, {\em J. Math. Anal. Appl.} 374(2) (2011) 516--528.

\bibitem[\bf Se2012]{Segura2012}
\textsc{J. Segura}, On bounds for solutions of monotonic first order difference-differential systems, {\em J. Inequal. Appl.} 2012 (2012) Art. 65.

\bibitem[\bf Sk1954]{Skovgaard1954}
\textsc{H. Skovgaard}, On inequalities of the Tur\'an type, {\em Math. Scand.} 2 (1954) 65--73.

\bibitem[\bf Sz\'a1950]{Szasz1950}
\textsc{O. Sz\'asz}, Inequalities concerning ultraspherical polynomials and Bessel functions, {\em Proc. Amer. Math. Soc.} 1 (1950) 256--267.

\bibitem[\bf Sz\'a1951]{Szasz1951}
\textsc{O. Sz\'asz}, Identities and inequalities concerning orthogonal polynomials and Bessel functions, {\em J. Analyse Math.} 1 (1951) 116–134.

\bibitem[\bf Sze1948]{Szego1948}
\textsc{G. Szeg\H{o}}, On an inequality of P. Tur\'an concerning Legendre polynomials, {\em Bull. Amer. Math. Soc.} 54 (1948) 401--405.

\bibitem[\bf TN1951]{ThiruvenkatacharNanjundiah1951}
\textsc{V.R. Thiruvenkatachar, T.S. Nanjundiah}, Inequalities concerning Bessel functions and orthogonal polynomials, {\em Proc. Indian Acad. Sci. Sect. A} 33 (1951) 373--384.

\bibitem[\bf Tu1950]{Turan1950}
\textsc{P. Tur\'an}, On the zeros of the polynomials of Legendre, {\em \v{C}asopis P\v{e}st. Mat. Fys.} 75 (1950) 113--122.

\bibitem[\bf Wa1944]{Watson1944}
\textsc{G.N. Watson}, {\em A Treatise on the Theory of Bessel Functions}, 2nd ed., Cambridge University Press, Cambridge, 1944.

\bibitem[\bf Wi1961]{Wilf1961}
\textsc{H.S. Wilf}, Subordinating factor sequences for convex maps of the unit circle, {\em Proc. Am. Math. Soc.} 12(5) (1961) 689--693.

\end{thebibliography}
\end{document}